\documentclass[11pt,twoside]{amsart}

\usepackage[latin1]{inputenc}
\usepackage[OT1]{fontenc}
\usepackage{amsmath}
\usepackage{amsthm}
\usepackage{amssymb}
\usepackage[all]{xy}
\usepackage{hyperref}
\usepackage{amscd}
\usepackage{a4wide}
\usepackage{enumitem}

\DeclareMathOperator{\id}{id}

\DeclareMathOperator{\Hom}{Hom}

\DeclareMathOperator{\sHom}{\mathcal{H}\textit{om}}
\DeclareMathOperator{\Ext}{Ext}

\DeclareMathOperator{\Coh}{Coh}

\DeclareMathOperator{\sExt}{\mathcal{E}\textit{xt}}
\DeclareMathOperator{\Ho}{H}

\DeclareMathOperator{\ord}{ord}

\DeclareMathOperator{\cone}{cone}

\DeclareMathOperator{\Sp}{\mathsf{Sp}}

\newcommand{\D}{{\rm D}}

\DeclareMathOperator{\Aut}{\mathsf{Aut}}

\DeclareMathOperator{\Pic}{\mathsf {Pic}}

\newcommand{\cH}{{\mathcal H}}

\newcommand{\pt}{\mathsf{pt}}

\newcommand{\IC}{\mathbb{C}}

\newcommand{\IN}{\mathbb{N}}
\newcommand{\IP}{\mathbb{P}}

\newcommand{\IR}{\mathbb{R}}
\newcommand{\IZ}{\mathbb{Z}}

\DeclareMathOperator{\FM}{\mathsf{FM}}

\DeclareMathOperator{\HK}{\mathsf{HK}}
\DeclareMathOperator{\CY}{\mathsf{CY}}
\DeclareMathOperator{\Kthree}{\mathsf{K3}}

\let\dim\relax
\DeclareMathOperator{\dim}{\mathsf{dim}}

\let\deg\relax
\DeclareMathOperator{\deg}{\mathsf{deg}}

\let\ord\relax
\DeclareMathOperator{\ord}{\mathsf{ord}}

\let\gcd\relax
\DeclareMathOperator{\gcd}{\mathsf{gcd}}

\newcommand{\sym}{\mathfrak S}

\newcommand{\cF}{\mathcal F}

\newcommand{\cE}{\mathcal E}

\newcommand{\cX}{\mathcal X}

\newcommand{\cQ}{\mathcal Q}

\newcommand{\reg}{\mathcal O}

\renewcommand{\P}{\mathbb P}
\renewcommand{\theta}{\vartheta}
\renewcommand{\phi}{\varphi}
\renewcommand{\_}{\underline{\,\,\,\,}}

\newtheorem{theorem}{Theorem}[section]
\newtheorem{prop}[theorem]{Proposition}
\newtheorem{lemma}[theorem]{Lemma}
\newtheorem{cor}[theorem]{Corollary}

\newtheorem{obs}[theorem]{Observation}

\theoremstyle{definition}
\newtheorem{definition}[theorem]{Definition}

\newtheorem{remark}[theorem]{Remark}
\newtheorem{convention}[theorem]{Convention}

\begin{document}

\title{Varieties with $\P$-units}
\author{Andreas Krug}
\address{Universit\" at Marburg}
\email{andkrug@mathematik.uni-marburg.de}
\begin{abstract}
We study the class of compact K\"ahler manifolds with trivial canonical bundle and the property that the cohomology of the trivial line bundle is generated by one element. If the square of the generator is zero, we get the class of strict Calabi--Yau manifolds. If the generator is of degree $2$, we get the class of compact hyperk\"ahler manifolds. We provide some examples and structure results for the cases where the generator is of higher nilpotency index and degree. In particular, we show that varieties of this type are closely related to higher-dimensional Enriques varieties.
\end{abstract}

\maketitle

\section{Introduction}

In this paper we will study a certain class of compact K\"ahler manifolds with trivial canonical bundle which contains all strict Calabi--Yau varieties as well as all hyperk\"ahler manifolds. 
For the bigger class of manifolds with trivial first Chern class $c_1(X)=0\in \Ho^2(X,\IR)$ there exists the following nice structure theorem, known as the \textit{Beauville--Bogomolov decomposition}; see \cite{Beasome}. Namely, each such manifold $X$ admits an \'etale covering $X'\to X$ which decomposes as
\begin{align*}
 X'=T\times \prod_{i} Y_i\times \prod_j Z_j
\end{align*}
where $T$ is a complex torus, the $Y_i$ are hyperk\"ahler, and the $Z_j$ are simply connected strict Calabi--Yau varieties of dimension at least $3$. 
 
Given a variety $X$, the graded algebra $\Ho^*(\reg_X):=\oplus_{i=0}^{\dim X}\Ho^i(X,\reg_X)[-i]$ is considered an important invariant; see, in particular, Abuaf \cite{Abuunit} who calls $\Ho^*(\reg_X)$ the \textit{homological unit} of $X$ and conjectures that it is stable under derived equivalences.
In this paper, we want to study varieties which have trivial canonical bundle and the property that the algebra $\Ho^*(\reg_X)$ is generated by one element.

The main motivation are the following two observations.
Let $X$ be a compact K\"ahler manifold.
\begin{obs}\label{CYobs}
 $X$ is a strict Calabi--Yau manifold if and only if the canonical bundle $\omega_X$ is trivial and $\Ho^*(\reg_X)\cong\IC[x]/x^2$ with $\deg x=\dim X$.  
These conditions can be summarised in terms of objects of the bounded derived category $\D(X):=\D^b(\Coh(X))$ of coherent sheaves. Namely, $X$ is a strict Calabi--Yau manifold if and only if $\reg_X\in \D(X)$ is a \textit{spherical object} in the sense of Seidel and Thomas \cite{ST}.
\end{obs}
 The above is a very simple reformulation of the standard definition of a strict Calabi--Yau manifold. The second observation is probably less well-known. 
\begin{obs}\label{HKobs}
 $X$ is a hyperk\"ahler manifold of dimension $\dim X=2n$ if and only if $\omega_X$ is trivial and $\Ho^*(\reg_X)\cong \IC[x]/x^{n+1}$ with $\deg x=2$. This is equivalent to the condition that $\reg_X\in \D(X)$ is a \textit{$\IP^n$-object} in the sense of Huybrechts and Thomas \cite{HT}. 
\end{obs}
Indeed, the structure sheaf of a hyperk\"ahler manifold is one of the well-known examples of a $\IP^n$-object; see \cite[Ex.\ 1.3(ii)]{HT}. The fact that $\Ho^*(\reg_X)$ also characterises the compact hyperk\"ahler manifolds follows from \cite[Prop.\ A.1]{HuyNW}. 

Inspired by this, we study the class of compact K\"ahler manifolds $X$ with the property that $\reg_X\in \D(X)$ is what we call a \textit{$\IP^n[k]$-object}; see Definition \ref{Pnkobjdef}. Concretely, this means:
\begin{enumerate}
 \item[(C1)] The canonical line bundle $\omega_X$ is trivial,
 \item[(C2)] There is an isomorphism of $\IC$-algebras $\Ho^*(\reg_X)\cong \IC[x]/x^{n+1}$ with $\deg x = k$.
\end{enumerate}
By Serre duality, such a manifold is of dimension $\dim X=\deg(x^n)=n\cdot k$.
For $n=1$, we get exactly the strict Calabi--Yau manifolds while for $k=2$ we get the hyperk\"ahler manifolds. 

In this paper, we will study the case of higher $n$ and $k$. We construct examples and prove some structure results.
If $\reg_X$ is a $\IP^n[k]$-object with $k>2$, the manifold $X$ is automatically projective; see Lemma \ref{algebraiclemma}. Hence, we will call $X$ a \textit{variety with $\IP^n[k]$-unit.}
The main results of this paper can be summarised as
\begin{theorem}\label{main}
Let $n+1=p^\nu$ be a prime power. Then the following are equivalent:
\begin{enumerate}
 \item There exists a variety with $\IP^n[4]$-unit,
\item There exists a variety with $\IP^n[k]$-unit for every even $k$,
\item There exists a strict Enriques variety of index $n+1$.
\end{enumerate}
For $n+1$ arbitrary, the implications (iii)$\Longrightarrow$(ii)$\Longrightarrow$(i) are still true.
\end{theorem}
We do not know whether or not (i)$\Longrightarrow$(iii) is true in general if $n+1$ is not a prime power, but we will prove a slightly weaker statement that holds for arbitrary $n+1$; see Section \ref{k=4Sec}.
In particular, the universal cover of a variety with $\IP^n[4]$-unit, with $n+1$ arbitrary, splits into a product of two hyperk\"ahler varieties; see Proposition \ref{k=4prop}.

Our notion of strict Enriques varieties is inspired by similar notions of higher dimensional analogues of Enriques surfaces due to Boissi\`ere, Nieper-Wi\ss kirchen, and Sarti \cite{BNWS} and Oguiso and Schr\"oer \cite{OSEnriques}.
There are known examples of strict Enriques varieties of index $3$ and $4$. Hence, we get
\begin{cor}\label{existencecor}
For $n=2$ and $n=3$ there are examples of varieties with $\IP^n[k]$-units for every even $k\in \IN$. 
\end{cor}

The motivation for this work comes from questions concerning derived categories and the notions are influenced by this. However, in this paper, with the exception Sections \ref{derivedsection} and \ref{autosection}, all results and proofs are also formulated without using the language of derived categories.

All our examples of varieties with $\IP^n[k]$-units are constructed using strict Enriques varieties or, equivalently, hyperk\"ahler varieties together with special automorphisms. It would be very interesting to find methods which allow to construct varieties with $\IP^n[k]$-units directly; maybe as moduli spaces of sheaves on varieties of dimension $k$ with trivial canonical bundle. By the above results, this could give rise to new examples of strict Enriques or even hyperk\"ahler varieties by considering the universal covers.

The paper is organised as follows. In Section \ref{notations}, we fix some notations and conventions. Sections \ref{derivedintro} and \ref{specialobjects} are a very brief introduction into derived categories and some types of objects that occur in these categories. In particular, we introduce the notion of $\IP^n[k]$-objects. 

In Section \ref{HKsection}, we say a few words about compact hyperk\"ahler manifolds. In Section \ref{autoaction}, we discuss automorphisms of Beauville--Bogomolov products and their action on cohomology. This is used in the following Section \ref{HKobssection} in order to give a proof of Observation \ref{HKobs}. This proof is probably a bit easier than the one in \cite[App.\ A]{HuyNW}. More importantly, it allows us to introduce some of the notations and ideas which are used in the later sections. In Section \ref{Enriquessection}, we discuss a class of varieties which we call strict Enriques varieties. There are two different notions of Enriques varieties in the literature (see \cite{BNWS} and \cite{OSEnriques}) and our notion is the intersection of these two; see Proposition \ref{classEnriques}\ref{eqofdef}. In Section \ref{Enriquesstacksection},  we quickly mention a generalisation; namely strict Enriques stacks. 

We give the definition of a variety with a $\IP^n[k]$-unit together with some basic remarks in Section \ref{Pnkdefsection}. Section \ref{nonexamplesection} provides two examples of varieties which look like promising candidates, but ultimately fail to have $\IP^n[k]$-units. In Section \ref{mainconstruction}, we construct series of varieties with $\IP^n[k]$-units out of strict Enriques varieties of index $n+1$. In particular, we prove the implication (iii)$\Longrightarrow$(ii) of Theorem \ref{main}.  

In Section \ref{generalfundamental}, we make some basic observations concerning the fundamental group and the universal cover of varieties with $\IP^n[k]$-units. In Section \ref{k=4Sec}, we specialise to the case $k=4$. We proof that the universal cover of a variety with $\IP^n[4]$-unit is the product of two hyperk\"ahler manifolds of dimension $2n$. Then we proceed to proof the implication (i)$\Longrightarrow$(iii) of Theorem \ref{main} for $n+1$ a prime power. 

Section \ref{furtherremarks} is a collection of some further observations and ideas. In Sections \ref{furtherconst1}, \ref{n=2}, and \ref{furtherconst3}, some further constructions leading to varieties with $\IP^n[k]$-units are discussed. We talk briefly about stacks with $\IP^n[k]$-units in Section \ref{Pnkstacks}. We see that symmetric quotients of strict Calabi--Yau varieties provide examples of stacks with $\IP^n[k]$-units for every $n$ and $k$.
In Section \ref{derivedsection}, we prove that the class of strict Enriques varieties is stable under derived equivalences, and in Section \ref{autosection} we study some derived autoequivalences of varieties with $\IP^n[k]$-units. In the final Section \ref{moduli}, we contemplate a bit about varieties with $\IP^n[k]$-units as moduli spaces.

\smallskip
\textbf{Acknowledgements.} The early stages of this work were done while the author was financially supported by the research grant KR 4541/1-1 of the DFG (German Research Foundation). He thanks Daniel Huybrechts, Marc Nieper-Wi\ss kirchen, S\"onke Rollenske, and Pawel Sosna for helpful discussions and comments.

\section{Notations and preliminaries}

\subsection{Notations and conventions}\label{notations}
\begin{enumerate}
\item Throughout, $X$ will be a compact K\"ahler manifold (often a smooth projective variety).
\item We denote the universal cover by $\widehat X\to X$.
\item\label{cc} If $\omega_X$ is of finite order $m$, we denote the canonical cover by $\pi\colon\widetilde X\to X$. It is defined by the properties that $\omega_{\widetilde X}$ is trivial and $\pi$ is an \'etale Galois cover of degree $m$. We have $\pi_*\reg_{\widetilde X}\cong \reg_X\oplus \omega_X^{-1}\oplus \omega_X^{-2}\oplus \dots \oplus \omega_X^{-(m-1)}$ and
the covering map $\widetilde X\to X$ is the quotient by a cyclic group $G=\langle g\rangle$ with $g\in \Aut(\widetilde X)$ of order $m$.
\item We will usually write graded vector spaces in the form $V^*=\oplus_{i\in \IZ}V^i[-i]$. The \textit{Euler characteristic} is given by the alternating sum $\chi(V^*)=\sum_{i\in \IZ}(-1)^i\dim V^i$.
\item Given a sheaf or a complex of sheaves $E$ and an integer $i\in \IZ$, we write $\Ho^i(X,E)$ for the $i$-th derived functor of global sections. In contrast, $\cH^i(E)$ denotes the cohomology of the complex in the sense kernel modulo image of the differentials.
\item We will usually write for short $\Ho^*(\reg_X)$ instead of $\Ho^*(X,\reg_X)$.
\item\label{HKitem} We write for short $Y\in \HK_{2d}$ to express the fact that $Y$ is a compact hyperk\"ahler manifold of dimension $2d$. In this case, we denote by $y$ a generator of $\Ho^2(\reg_Y)$, i.e.\ $y$ is the complex conjugate of a symplectic form on $Y$. If we just write $Y\in \HK$, this means that $Y$ is a hyperk\"ahler manifold of unspecified dimension. Sometimes, we write $Y\in \Kthree$ instead of $Y\in \HK_2$.
\item We write for short $Z\in \CY_{e}$ to express the fact that $Z$ is a compact simply connected strict Calabi--Yau variety of dimension $e\ge 3$. In this case, we denote by $z$ a generator of $\Ho^e(\reg_Z)$, i.e.\ $z$ is the complex conjugate of a volume form on $Z$. If we just write $Z\in \CY$, this means that $Z$ is a simply connected strict Calabi--Yau variety of unspecified dimension.
\item We denote the connected zero-dimensional manifold by $\pt$.
\item For $n\in \IN$, we denote the symmetric group of permutations of the set $\{1,\dots,n\}$ by $\sym_n$. Given a space $X$ and a permutation $\sigma\in \sym_n$, we denote the automorphism of the cartesian product $X^n$ which is given by the according permutation of components again by $\sigma\in \Aut(X^n)$.  
\item For $n\in \IN$, we denote by $\mu_n\subset \IC^*$ the cyclic group of $n$-th roots of unity.
\item If we write $i\neq j$ as a subscript of a sum, we mean that the sum is indexed by all unordered tuples of distinct $i$ and $j$ (in some index set which is, hopefully, clear from the context). Similarly, a sum $\sum_{i_1\neq i_2\neq \dots\neq i_\ell}$ is meant to summarise terms indexed by unordered $\ell$-tuples of pairwise distinct elements. 
\end{enumerate}
\subsection{Derived categories of coherent sheaves}\label{derivedintro}
As mentioned in the introduction, knowledge of derived categories is not necessary for the understanding of this paper. However, often things can be stated in the language of derived categories in the most convenient way, and questions concerning derived categories motivated this work. Hence, we will give, in a very brief form, some basic definitions and facts.

The derived category is defined as the category of complexes of coherent sheaves localised at the class of quasi-isomorphisms. Hence, the objects of $\D(X)$ are (bounded) complexes of coherent sheaves. The morphisms are morphisms of complexes together with formal inverses of quasi-isomorphisms. In particular, every quasi-isomorphism between complexes becomes an isomorphism in $\D(X)$. The derived category $\D(X)$ is a triangulated category. In particular, there is the shift autoequivalence $[1]\colon \D(X)\to \D(X)$. Given two objects $E,F\in \D(X)$, there is a graded Hom-space $\Hom^*(E,F)=\oplus_i \Hom_{\D(X)}(E,F[i])[-i]$. For $E=F$, this is a graded algebra by the Yoneda product (composition of morphisms). There is a fully faithful embedding $\Coh(X)\hookrightarrow \D(X)$, $A\mapsto A[0]$ which is given by considering sheaves as complexes concentrated in degree zero. Most of the time, we will denote $A[0]$ simply by $A$ again. For $A, B\in \Coh(X)$, we have $\Hom^*(A,B)\cong \Ext^*(A,B)$. Besides the shift functor, the data of a triangulated category consists of a class of distinguished triangles $E\to F\to G\to E[1]$ consisting of objects and morphisms in $\D(X)$ satisfying certain axioms. In particular, every morphism $f\colon E\to F$ in $\D(X)$ can be completed to a distinguished triangle \[E\xrightarrow{f} F\to G\to E[1]\,.\] The object $G$ is determined by $f$ up to isomorphism and denoted by $G=\cone(f)$. There is a long exact cohomology sequence
\[
 \dots\to \cH^{i-1}(\cone(f))\to \cH^i(E)\to \cH^i(F)\to \cH^i(\cone(f))\to \cH^{i+1}(E)\to \dots\,.
\]
\subsection{Special objects of the derived category}\label{specialobjects}
In the following, we will recall the notions of exceptional, spherical and $\IP$-objects in the derived category $\D(X)$ of coherent sheaves on a compact K\"ahler manifold $X$. Exceptional objects can be used in order to decompose derived categories while spherical and $\IP$-objects induce autoequivalences. Our main focus in this paper, however, will be to characterise varieties where $\reg_X\in \D(X)$ is an object of one of these types.

\begin{definition}\label{exceptionalobj}
An object $E\in\D(X)$ is called \textit{exceptional} if $\Hom^*(E,E)\cong \IC[0]$.  
\end{definition}
Let $X$ be a Fano variety, i.e.\ the anticanonical bundle $\omega_X^{-1}$ is ample.
Then, by Kodaira vanishing, every line bundle on $X$ is exceptional when considered as an object of the derived category $\D(X)$; see also Remark \ref{regspecial}. Similarly, every line bundle on an Enriques surface is exceptional. Another typical example of an exceptional object is the structure sheaf $\reg_C\in \D(S)$ of a $(-1)$-curve $\IP^1\cong C\subset S$ on a surface.  
 \begin{definition}[\cite{ST}]
 An object $E\in\D(X)$ is called \textit{spherical} if 
 \begin{enumerate}
  \item $E\otimes \omega_X\cong E$,
  \item $\Hom^*(E,E)\cong \IC[0]\oplus \IC[\dim X]\cong \Ho^*(\mathbb S^{\dim X}, \IC)$.
 \end{enumerate}
 \end{definition}
Every line bundle on a strict Calabi--Yau variety is spherical. Another typical 
example of a spherical object is the structure sheaf $\reg_C\in \D(S)$ of a $(-2)$-curve $\IP^1\cong C\subset S$ on a surface.  
\begin{definition}[\cite{HT}]
Let $n\in \IN$. An object $E\in\D(X)$ is called \textit{$\IP^n$-object} if 
\begin{enumerate}
  \item $E\otimes \omega_X\cong E$,
  \item There is an isomorphism of $\IC$-algebras $\Hom^*(E,E)\cong \IC[x]/x^{n+1}$ with $\deg x=2$.
 \end{enumerate}
 \end{definition}
Condition (ii) can be rephrased as $\Hom^*(E,E)\cong \Ho^*(\mathbb P^{n}, \IC)$. As we will see in the next subsection, every line bundle on a compact hyperk\"ahler manifold is a $\IP$-object. Another typical example is the structure sheaf of the centre of a Mukai flop. 
\begin{definition}\label{Pnkobjdef}
 Let $n,k\in \IN$. An object $E\in \D(X)$ is called \textit{$\IP^n[k]$-object} if 
\begin{enumerate}
\item $E\otimes \omega_X\cong E$,
\item There is an isomorphism of $\IC$-algebras $\Hom^*(E,E)\cong \IC[x]/x^{n+1}$
with $\deg x=k$.
\end{enumerate} 
\end{definition}
\begin{remark}\label{Pobjremark}
If there is a $\IP^n[k]$-object $E\in \D(X)$, we have $\dim X=n\cdot k$ by Serre duality.
\end{remark}
\begin{remark}
For $n=1$, the $\IP^1[k]$-objects coincide with the spherical objects. For $k=2$, the $\IP^n[2]$-objects are exactly the $\IP^n$-objects in the sense of Huybrechts and Thomas. 
\end{remark}

The names spherical and $\IP$-objects come from the fact that their graded endomorphism algebra coincides with the cohomology of spheres and projective spaces, respectively. Hence, it would be natural to name $\IP^n[k]$-object by series of manifolds whose cohomology is of the form $\IC[x]/x^{n+1}$ with $\deg x=k$. For $k=4$, there are the quaternionic projective spaces. For $k>4$, however, there are no such series. Hence, we will stick to the notion of $\IP^n[k]$-objects which is justified by the following 
\begin{remark}\label{inducedautoremark}
A $\IP^n[k]$-object is essentially the same as a $\IP$-functor (see \cite{Add}) $\D(\pt)\to \D(X)$ with $\IP$-cotwist $[-k]$. In particular, as we will further discuss in Section \ref{autosection}, it induces an autoequivalence of $\D(X)$.
\end{remark}

\begin{remark}\label{regspecial}
Given a compact K\"ahler manifold $X$, the following are equivalent:
\begin{enumerate}
 \item $\reg_X$ is exceptional (a $\IP^n[k]$-object).
 \item Every line bundle on $X$ is exceptional (a $\IP^n[k]$-object).
 \item Some line bundle on $X$ is exceptional (a $\IP^n[k]$-object). 
\end{enumerate}
Indeed, for every line bundle $L$ on $X$, we have isomorphisms of $\IC$-algebras
\[\Hom^*(L,L)\cong \Hom^*(\reg_X,\reg_X)\cong \Ho^*(\reg_X)\] where the latter is an algebra by the cup product. Furthermore, $L\otimes \omega_X\cong L$ holds if and only if $\omega_X$ is trivial.
\end{remark}

\section{Hyperk\"ahler and Enriques varieties}
In this section, we first review some results on hyperk\"ahler manifolds and their automorphisms. In particular, we give a proof of Observation \ref{HKobs}, i.e.\ the fact that hyperk\"ahler manifolds can be characterised by the property that the trivial line bundle is a $\IP$-object. Then we introduce and study strict Enriques varieties. They are a generalisation of Enriques surfaces to higher dimensions and can be realised as quotients of hyperk\"ahler varieties. 

\subsection{Hyperk\"ahler manifolds}\label{HKsection}

Let $X$ be a compact K\" ahler manifold of dimension $2n$. We say that $X$ is \textit{hyperk\"ahler} if and only if its Riemannian holonomy group is the symplectic group $\Sp(n)$.   
A compact K\"ahler manifold $X$ is hyperk\"ahler if and only if it is \textit{irreducible holomorphic symplectic} which means that it is simply connected and $\Ho^2(X,\wedge^2\omega_X)$ is spanned by an everywhere non-degenerate $2$-form, called \textit{symplectic form}; see e.g.\ \cite{HuyHKbook}.

The structure sheaf of a hyperk\"ahler manifold is a $\IP^n$-object; see \cite[Ex.\ 1.3(ii)]{HT}. This means that the canonical bundle $\omega_X=\wedge^{2n}\Omega_X$ is trivial and $\Ho^*(\reg_X)=\IC[x]/x^{n+1}$; compare Item \ref{HKitem} of Section \ref{notations}. This follows essentially from the holonomy principle together with Bochner's principle. We will see in Section \ref{HKobssection} that also the converse holds, which amounts to Observation \ref{HKobs}.


\subsection{Automorphisms and their action on cohomology}\label{autoaction}
In the later sections, we will often deal with automorphisms of Beauville--Bogomolov covers. 
There is the following result of Beauville \cite[Sect.\ 3]{Beasome}. 
\begin{lemma}\label{productauto}
 Let $X'\cong \prod_i Y_i^{\lambda_i}\times \prod_j Z_j^{\nu_j}$ be a finite product with $Y_i\in \HK_{2d_i}$ and $Z_j\in \CY_{e_j}$ such that the $Y_i$ and $Z_j$ are pairwise non-isomorphic. Then, every automorphism of $X'$ preserves the decomposition up to permutation of factors. More concretely, every automorphism $f\in \Aut(X')$ is of the form $f=\prod f_{Y_i^{\lambda_i}} \times\prod f_{Z_j^{\nu_j}}$ with $f_{Y_i^{\lambda_i}}\in \Aut(Y_i^{\lambda_i})$ and $f_{Z_j^{\nu_j}}\in \Aut(Z_j^{\nu_j})$. Furthermore, $f_{Y_i^{\lambda_i}}=(f_{Y_{i1}}\times \dots \times f_{Y_{i\lambda_i}})\circ \sigma_{Y_i,f}$ with $f_{Y_{i\alpha}}\in \Aut(Y_i)$ and $\sigma_{Y_i,f}\in \sym_{\lambda_i}$. Similarly, $f_{Z_j^{\nu_i}}=(f_{Z_{j1}}\times \dots \times f_{Z_{i\nu_i}})\circ \sigma_{Z_j,f}$ with $f_{Z_{j\beta}}\in \Aut(Z_i)$ and $\sigma_{Z_j,f}\in \sym_{\nu_i}$. 
\end{lemma}
Let $X'\cong \prod_i Y_i^{\mu_i}\times \prod_j Z_j^{\nu_j}$ as above. For $\alpha=1,\dots, \mu_i$ we denote by $y_{i\alpha}\in \Ho^2(\reg_{X'})$ the image of $y_i\in \Ho^2(\reg_{Y_i})$ under pull-back along the projection $X'\to Y_i$ to the $\alpha$-th $Y_i$ component; compare Item \ref{HKitem} of Section \ref{notations}. For $\beta=1,\dots,\nu_j$, the class $z_{j\beta}$ is defined analogously. By the K\"unneth formula, the $y_{i\alpha}$ and $z_{j\beta}$ together generate the cohomology $\Ho^*(\reg_{X'})$ and we have
\begin{align}\label{productcoh}
\Ho^*(\reg_{X'})=\IC[\{y_{i\alpha} \}_{i\alpha}, \{z_{j\beta}\}_{j\beta}]/(y_{i\alpha}^{d_i}, z_{j\beta}^2)\,.
\end{align}
Let $Y\in \HK$. The action of automorphisms on $\Ho^2(\reg_X)\cong \IC$ defines a group character which we denote by
\[
 \rho_Y\colon \Aut(Y)\to \IC^*\quad ,\quad f\mapsto \rho_{Y,f}\,. 
\]
In particular, an automorphism $f\in \Aut(Y)$ of finite order $\ord f=m$ acts on $\Ho^2(\reg_X)$ by multiplication by an $m$-th root of unity $\rho_{Y,f}\in \mu_m$. Similarly, for $Z\in \CY_k$ we have a character $\rho_Z\colon \Aut(Z)\to \IC^*$ given by the action of automorphisms on $\Ho^k(\reg_Z)$.  

\begin{cor}\label{inducedaction}
Let $f\in \Aut(X')$ be of finite order $d$. Then the induced action of $f$ on cohomology is given by permutations of the $y_{i\alpha}$ with fixed $i$ and the $z_{j\beta}$ with fixed $j$ together with multiplications by $d$-th roots of unity. This means
\[f\colon \quad y_{i\alpha}\mapsto \rho_{Y_{i\alpha}, f_{Y_{i\alpha}}}\cdot y_{i\sigma_{Y_i,f}(\alpha)}\quad, \quad z_{j\beta}\mapsto \rho_{Z_{j\beta},f_{Z_{j\beta}}}\cdot z_{j\sigma_{Z_j,f}(\beta)}\]
with $\rho_{Y_{i\alpha},f},\rho_{Z_{j\beta},f}\in \mu_d$.
\end{cor}

The main takeaway for the computations in the latter sections is that the cohomology classes can only be permuted if the corresponding factors of the product coincide.  

\begin{definition}
Let $Y\in \HK$ and $f\in \Aut Y$ of finite order. We call the order of $\rho_{Y,f}\in\IC$ the \textit{symplectic order} of $f$. The reason for the name is that $f$ acts by a root of unity of the same order, namely $\bar\rho_{Y,f}$, on $\Ho^0(\wedge^2\Omega_X)$, i.e.\ on the symplectic forms. 
We say that $f$ is \textit{symplectic} if $\rho_{Y,f}=1$.
In general, the symplectic order divides the order of $f$ in $\Aut(X)$. We say that $f$ is \textit{purely non-symplectic} if its symplectic order is equal to $\ord f$.  
\end{definition}


\begin{lemma}\label{HKfreeauto}
Let $Y\in \HK_{2n}$ and let $f\in \Aut(Y)$ be an automorphism of finite order $m$ such that the generated group $\langle f\rangle$ acts freely on $Y$. Then $f$ is purely non-symplectic and $m\mid n+1$. Similarly, every fixed point free automorphism of finite order of a strict Calabi--Yau variety is a non-symplectic involution. 
\end{lemma}
\begin{proof}
This follows from the holomorphic Lefschetz fixed point theorem; compare \cite[Sect.\ 2.2]{BNWS}.
\end{proof}

\begin{cor}\label{quotientcanonical}
Let $Y\in \HK_{2n}$ and let $X=Y/\langle f\rangle$ be the quotient by a cyclic group of automorphisms acting freely. Then $\omega_X$ is non-trivial and of finite order.
\end{cor}
\begin{proof}
 The order of $\omega_X$ is exactly the order of the action of $f$ on $\Ho^{2n}(\reg_X)$, i.e. the order of $\rho_{Y,f}^n\in \IC^*$. By the previous lemma, this order is finite and greater than one. 
\end{proof}

Here is a simple criterion for automorphisms of products to be fixed point free.
\begin{lemma}\label{freeautos}
\begin{enumerate}
 \item\label{freeautos1} Let $X_1,\dots, X_k$ be manifolds and $f_i\in \Aut(X_i)$. Then \[f_1\times \dots\times f_k \in \Aut(X_1\times\dots\times X_k)\] is fixed point free if and only if at least one of the  $f_i$ is fixed point free.
 \item\label{freeautos2} Let $X$ be a manifold and $g_1,\dots, g_k\in \Aut(X)$. Consider the automorphism \[\phi=(g_1\times\dots \times g_k)\circ (1\,\, 2\,\dots\, k)\in \Aut(X^k)\] given by $(p_1, p_2, \dots ,p_k)\mapsto (g_1(p_k), g_2(p_1), \dots, g_k(p_{k-1}))$. Then $\phi$ is fixed point free if and only if the composition $g_k\circ g_{k-1}\circ \dots \circ g_1$ (or, equivalently, $g_i\circ g_{i-1}\circ \dots\circ g_{i+1}$ for some $i=1,\dots k$) is fixed point free. 
\end{enumerate}
\end{lemma}
We also will frequently use the following well-known fact.
\begin{lemma}\label{EulerGlemma}
 Let $X'$ be a smooth projective variety and let $G\subset \Aut(X')$ be a finite subgroup which acts freely. Then, the quotient variety $X:=X'/G$ is again smooth projective and 
 \[
\chi(\reg_{X'})=\chi(\reg_X)\cdot \ord G\,.  
 \]
 Furthermore, $\Ho^*(\reg_X)=\Ho^*(\reg_{X'})^G$. 
\end{lemma}

\subsection{Proof of Observation \ref{HKobs}}\label{HKobssection}

We already remarked in Section \ref{HKsection} that the structure sheaf of a hyperk\"ahler manifold is a $\IP$-object. Hence, for the verification of Observation \ref{HKobs} we only need to prove the following \begin{prop}\label{HKchar}
 Let $X$ be a compact K\"ahler manifold such that $\reg_X\in \D(X)$ is a $\IP^n[2]$-object. Then $X$ is hyperk\"ahler of dimension $2n$.  
\end{prop}
\begin{proof}
As already mentioned in the introduction, this follows immediately from \cite[Prop.\ A.1]{HuyNW}. We will give a slightly different proof.

Recall that the assumption that $\reg_X$ is a $\IP^n$-object means
\begin{enumerate}
 \item $\omega_X$ is trivial,
 \item $\Ho^*(\reg_X)\cong \IC[x]/x^{n+1}$ with $\deg x=2$.
\end{enumerate}
For $n=1$, it follows easily by the Kodaira classification of surfaces, that $X\in \Kthree=\HK_{2}$. Hence, we may assume that $n\ge 2$.

Assumption (i) says that, in particular, $c_1(X)=0$. Hence, we have an \'etale cover $X'\to X$ and a Beauville--Bogomolov decomposition 
\begin{align}\label{coverdec}
  X'=T\times \prod_{i} Y_i\times \prod_j Z_j\,.
\end{align}
The plan is to show that $X'$ is hyperk\"ahler and the cover is an isomorphism.
\begin{convention}\label{BBcoverconv}
Whenever we have a Beauville--Bogomolov decomposition of the form (\ref{coverdec}), $T$ is a complex torus, $Y_i\in \HK_{2d_i}$ is a hyperk\"ahler of dimension $2d_i$ and $Z_j\in \CY_{e_j}$ is a strict simply connected Calabi--Yau variety of dimension $e_j\ge 3$. Furthermore, $\Ho^2(\reg_{Y_i})=\langle y_i\rangle$ and $\Ho^{e_j}(\reg_{Z_j})=\langle z_j\rangle$.   
\end{convention}
By Assumption (ii), we have $\chi(\reg_X)=n+1$. On the other hand, since $X'\to X$ is \'etale, say of degree $m$, we have 
\begin{align}\label{Eulerformula}
m(n+1)=m\cdot \chi(\reg_X)=\chi(\reg_{X'})=\chi(T)\cdot \prod_i\chi(\reg_{Y_i})\cdot \prod_j\chi(\reg_{Z_j})\,.
\end{align}
This implies that $T=\pt$ and all $e_j$ are even. Otherwise, the right-hand side of (\ref{Eulerformula}) would be zero. Since the torus part is trivial, $X'$ is simply connected. Hence, $X'=\widehat X$ is the universal cover of $X=\widehat X/G$ where $\pi_1(X)\cong G\subset \Aut(X)$. It follows by Lemma \ref{EulerGlemma} that \[\IC[x]/x^{n+1}\cong \Ho^*(\reg_X)\cong \Ho^*(\reg_{\widehat X})^G\subset \Ho^*(\reg_{\widehat X})\,.\] In particular, there must be an $x\in\Ho^2(\reg_{\widehat X})^G\subset \Ho^2(\reg_{\widehat X})$ such that $0\neq x^n\in \Ho^{2n}(\reg_{\widehat X})$. Since \[2n=\dim X=\dim \widehat X=\sum_i2d_i+\sum_je_j\,,\]
we have $\Ho^{2n}(\reg_{\widehat X})=\langle s \rangle$ with $s=\prod_i y_i^{d_i}\cdot \prod_j z_j$. As $\deg x=2$ and $\deg z_j=e_j\ge 3$, it follows that $x^n$ can be a non-zero multiple of $s$ only if $X'$ does not have Calabi--Yau factors. This means that $X=\prod_iY_i$ and $x=\sum_i y_i$ (up to coefficients which we can absorb by the choice of the generators $y_i$ of $\Ho^2(\reg_{Y_i})$). Every element of $G$ acts by some permutation on the $y_i$; see Corollary \ref{inducedaction}. By assumption, $\Ho^2(\reg_X)$ is of dimension one. Hence, $\Ho^2(\reg_{\widehat X})^G=\langle x\rangle$. It follows that the action of $G$ on the $y_i$ is transitive. Otherwise, there would be $G$-invariant summands of $x=\sum_i y_i$ which would be linearly independent. Hence, again by Corollary \ref{inducedaction}, we have 
$\widehat X\cong Y^\ell$ for some $Y\in \HK_{2d}$. 
For dimension reasons, $d\cdot \ell=n$. 

We assume for a contradiction that $\ell>1$. We have the $G$-invariant class
\begin{align}\label{t2}x^2=\sum_\alpha y_\alpha^2 + 2\sum_{\alpha\neq\beta} y_\alpha y_\beta\in \Ho^4(\reg_{X'})^G=\Ho^4(\reg_X)\,.\end{align}
It follows by Corollary \ref{inducedaction} that the two summands in (\ref{t2}) are again $G$-invariant.
But, by assumption, $h^4(\reg_{X})=1$. Thus, one of the two summands must be zero. By (\ref{productcoh}), we see that the only possibility for this to happen is $d=1$, i.e. $Y\in \Kthree$. Thus, $\ell=n$. 
Note that $\ord G=\deg(X'\to X)=m$. By (\ref{Eulerformula}) or Lemma \ref{EulerGlemma}, we have $m\mid \chi(X')=\chi(Y)^n=2^n$.
As $G$ acts transitively on $\{y_1,\dots,y_n\}$ we get $n\mid m\mid 2^n$. Again by (\ref{Eulerformula}), also $n+1\mid 2^n$. For $n\ge 2$, this is a contradiction.   

Hence, we are in the case $\ell=1$ which means that $\hat X=Y\in \HK_{2n}$. In particular, $\chi(\reg_{\hat X})=n+1=\chi(X)$. By (\ref{Eulerformula}), we get $m=1$ which means that we have an isomorphism $Y\cong X$.
\end{proof}

\subsection{Enriques varieties}\label{Enriquessection}

In this section we will consider a certain class of compact K\"ahler manifolds with the property that $\reg_X\in \D(X)$ is exceptional; see Definition \ref{exceptionalobj}. These manifolds are automatically algebraic by the following result; see e.g.\ \cite[Exc.\ 7.1]{VoisinI}.  
\begin{lemma}\label{algebraiclemma}
 Let $X$ be a compact K\"ahler manifold with $\Ho^2(\reg_X)=0$. Then $X$ is projective.
\end{lemma}

From now on, let $E$ be a smooth projective variety.

\begin{definition}\label{strictEnriquesdef}
We call $E$ a \textit{strict Enriques variety} if the following three conditions hold:
\begin{enumerate}[label=(S\arabic*)]
 \item\label{S1} The trivial line bundle $\reg_E$ is exceptional.
 \item\label{S2} The canonical line bundle $\omega_E$ is non-trivial and of finite order $m:=\ord(\omega_E)$ in $\Pic E$ (this order is called the \textit{index} of $E$).
 \item\label{S3} The canonical cover $\widetilde E$ of $E$ is hyperk\"ahler.
\end{enumerate}
\end{definition}

This definition is inspired by similar, but different, notions of higher-dimensional Enriques varieties which are as follows.

\begin{definition}[\cite{BNWS}]
 We call $E$ a \textit{BNWS (Boissi\`ere--Nieper-Wi\ss kirchen--Sarti) Enriques variety} if the following three conditions hold:
\begin{enumerate}[label=(BNWS\arabic*)]
 \item $\chi(\reg_E)=1$.
 \item The canonical line bundle $\omega_E$ is non-trivial and of finite order $m:=\ord(\omega_E)$ in $\Pic E$ (this order is called the \textit{index} of $E$).
 \item The fundamental group of $E$ is cyclic of the same order, i.e. $\pi_1(E)\cong \mu_m$.
\end{enumerate}
\end{definition}

\begin{definition}[\cite{OSEnriques}]
We call $E$ an \textit{OS (Oguiso--Schr\"oer) Enriques variety} if $E$ is not simply connected and its universal cover $\widehat E$ is a compact hyperk\"ahler manifold.
\end{definition}

\begin{prop}\label{classEnriques}
\begin{enumerate}
 \item\label{prop1} Let $E$ be a strict Enriques variety of index $n+1$. Then $\dim E=2n$.
\item\label{prop2} Conversely, every smooth projective variety $E$ satisfying (S2) with $m=n+1$, (S3), and
 $\dim E=2n$ is already a strict Enriques variety.
  \item\label{Enriquesquotient} Strict Enriques varieties of index $n+1$ are exactly the quotient varieties of the form $E=Y/\langle g\rangle$, where $Y\in \HK_{2n}$ and $g\in \Aut(Y)$ is purely symplectic of order $n+1$ such that $\langle g\rangle$ acts freely on $Y$. 
 \item\label{eqofdef} $X$ is a strict Enriques variety if and only if it is BNWS Enriques and OS Enriques.
 \end{enumerate}
\end{prop}
\begin{proof}
Let $E$ be a strict Enriques variety of index $n+1$ with canonical cover $\widetilde E\in \HK_{2d}$. To verify \ref{prop1} we have to show that $d=n$.
By definition of the canonical cover (see Section \ref{notations} \ref{cc}), the covering map $\widetilde E\to E$ is the quotient by a cyclic group $G$ of order $n+1$. 
As $\widetilde E\in \HK_{2d}$, we have $\chi(\reg_Y)=d+1$. Also, $\chi(\reg_E)=1$ by \ref{S1}. 
We get $d=n$ by Lemma \ref{EulerGlemma}. 

Consider now a smooth projective variety $E$ with $\ord \omega_E=n+1$ and $\dim E=2n$ such that its canonical cover $\widetilde E$ is hyperk\"ahler, necessarily of $\dim \widetilde E=\dim E=2n$. Then, again by Lemma \ref{EulerGlemma}, we have $\chi(\reg_E)=1$. Furthermore, 
\begin{align}\label{Hochain}\IC[0]\subset \Ho^*(\reg_{E})\cong\Ho^*(\reg_{\widetilde E})^G\subset \Ho^*(\reg_{\widetilde E})\cong \IC[y]/y^{n+1}\end{align}
with $\deg y=2$. In order to get $\chi(\reg_E)=1$, the first inclusion must be an equality which means that $\reg_E$ is exceptional. 

Let us proof part \ref{Enriquesquotient}. Given a strict Enriques variety $E$ of index $n+1$ the canonical cover $Y:=\widetilde E$ has the desired properties.

Conversely, let $Y\in \HK_{2n}$ together with a purely non-symplectic $g\in \Aut(Y)$ of order $n+1$ such that $\langle g\rangle$ acts freely on $Y$, and set $E:=Y/\langle g\rangle$. The action of $g$ on the cohomology $\Ho^*(\reg_{Y})=\IC[y]/y^{n+1}$ is given by $g\cdot y^i=\rho_{Y,g}^iy^i$. Since, by assumption, $\rho_{Y,g}$ is a primitive $(n+1)$-th root of unity, we get $\Ho^*(\reg_E)\cong \Ho^*(\reg_Y)^G\cong \IC[0]$, hence \ref{S1}. The action of $g$ on the $n$-th power of a symplectic form, hence on the canonical bundle $\omega_Y$, is also given by multiplication by $\rho_{Y,g}$. It follows that the canonical bundle $\omega_E$ of the quotient is of order $n+1$ and $Y\to E$ is the canonical cover.

For the proof of \ref{eqofdef}, first note that (S1) implies (BNWS1). Furthermore, given a strict Enriques variety $E$, the canonical cover $Y=\widetilde E$ of $E$ is also the universal cover, since $Y$ is connected. From this, we get (BNWS2) and (BNWS3). Furthermore, $E$ is OS Enriques, since $Y$ is hyperk\"ahler.

Conversely, if $E$ is BNWS and OS Enriques, its canonical and universal cover coincide and is given by a hyperk\"ahler manifold $Y$ with the properties as in \ref{Enriquesquotient}.
\end{proof}

Note that the variety $Y\in \HK_{2n}$ from part \ref{Enriquesquotient} of the proposition is the universal as well as the canonical cover of $E$. We call $Y$ the \textit{hyperk\"ahler cover} of $E$.

Another way to characterise strict Enriques varieties is as OS Enriques varieties whose fundamental group have the maximal possible order; see \cite[Prop.\ 2.4]{OSEnriques}.

Strict Enriques varieties of index $2$ are exactly the Enriques surfaces.
To get examples of higher index, by part \ref{eqofdef} of the previous proposition, we just have to look for examples which occur in \cite{BNWS} as well as in \cite{OSEnriques}.

\begin{theorem}[\cite{BNWS},\cite{OSEnriques}]\label{Enriquesexistence}
There are strict Enriques varieties of index $2$, $3$, and $4$. 
\end{theorem}
Note that the statement does not exclude the existence of strict Enriques varieties of index greater than $4$, but, for the time being, there are no known examples.

In the known examples of index $n+1=3$ or $n+1=4$, the hyperk\"ahler cover $Y$ is given by a generalised Kummer variety $K_nA\subset A^{[n+1]}$. More concretely, in these examples $A$ is an abelian surface isogenous to a product of elliptic curves with complex multiplication, and there is a non-symplectic automorphism $f\in \Aut(A)$ of order $n+1$ which induces a non-symplectic fixed point free automorphism $K_n(f)\in \Aut(K_nA)$ of the same order.      

Note that there are examples of varieties which are BNWS Enriques but not OS Enriques \cite[Sect.\ 4.3]{BNWS} and of the converse \cite[Sect.\ 4]{OSEnriques}.

We will use the following lemma in the proof of Theorem \ref{existencethm}. 
\begin{lemma}\label{Enriquescanonical}
Let $E$ be a strict Enriques variety of index $n+1$ with hyperk\"ahler cover $Y$. Then there is an isomorphism of algebras $\oplus_{s=0}^n\Ho^*(\omega_E^{-s})\cong \Ho^*(\reg_Y)=\IC[y]/{y^{n+1}}$. Under this isomorphism, $\Ho^*(\omega_E^{-s})\cong \IC\cdot y^s\cong \IC[-2s ]$.
\end{lemma}
\begin{proof}
Let $\pi\colon Y\to E$ be the morphism which realises $Y$ as the universal and canonical cover of $E$. By the construction of the canonical cover (see Section \ref{notations} \ref{cc}), we have an isomorphism of $\reg_E$-algebras $\pi_*\omega_Y\cong \reg_E\oplus \omega_E^{-1}\oplus \dots \oplus \omega_E^{-n}$. Hence, we get an isomorphism of graded $\IC$-algebras
\begin{align}\label{regdecomposition}
\IC[y]/y^{n+1}\cong \Ho^*(\reg_Y)\cong \Ho^*(\reg_E)\oplus \Ho^*(\omega_E^{-1})\oplus \dots \oplus \Ho^*(\omega_E^{-n})
\end{align}
with $\deg y=2$. Hence, for the proof of the assertion, it is only left to show that the generator $y$ lives in the direct summand $\Ho^*(\omega_E^{-1})$ under the decomposition (\ref{regdecomposition}). This follows from the fact that $\omega_E^{-n}=\omega_E$, so by Serre duality $\Ho^*(\omega_E^{-n})=\IC[-2n]$.    
\end{proof}
\subsection{Enriques stacks}\label{Enriquesstacksection}
The main difficulty in finding pairs $Y\in \HK$ and $f\in \Aut(Y)$ which, by Proposition \ref{classEnriques}\ref{Enriquesquotient}, induce strict Enriques varieties, is the condition that $\langle f\rangle$ acts freely.

Let us drop this assumption and consider a $Y\in \HK_{2n}$ together with a non-symplectic automorphism $f\in \Aut(Y)$ which may have fixed points. Then we call the corresponding quotient stack $\cE=[Y/\langle f\rangle]$ a \textit{strict Enriques stack}. In analogy to the proof of Proposition \ref{classEnriques}, one can show that there is also the following equivalent
\begin{definition}
A \textit{strict Enriques stack} is a smooth projective orbifold $\cE$ such that 
\begin{enumerate}[label=(S\arabic*')]
 \item The trivial line bundle $\reg_\cE$ is exceptional.
 \item The canonical line bundle $\omega_\cE$ is non-trivial and of finite order $m:=\ord(\omega_\cE)$ in $\Pic \cE$ (this order is called the \textit{index} of $\cE$).
 \item The canonical cover $\widetilde \cE$ of $\cE$ is a hyperk\"ahler manifold of dimension $\dim \widetilde E=\dim E=2(m-1)$.
\end{enumerate}
\end{definition}
Note that, in contrast to the case of strict Enriques varieties, the formula relating index and dimension is not a consequence of the other conditions but is part of the assumptions.

%
%

As alluded to above, it is much easier to find examples of strict Enriques stacks compared to strict Enriques varieties. Let $S\in \Kthree$ together with a purely non-symplectic automorphism $f\in \Aut(S)$ of order $n+1$ (which may, and, for $n+1>2$, will have fixed points). Then the quotient of the associated Hilbert scheme of points by the induced automorphism $[X^{[n]}/f^{[n]}]$ is a strict Enriques stack. There are also examples of strict Enriques stacks whose hyperk\"ahler cover is $K_5(A)$; compare \cite[Rem. 4.1]{BNWS}.


\section{Construction of varieties with $\IP^n[k]$-units}
\subsection{Definition and basic properties}\label{Pnkdefsection}
\begin{definition}\label{Pnkdef}
Let $X$ be a compact K\"ahler manifold. We say that $X$ has a \textit{$\IP^n[k]$-unit} if $\reg_X$ is a $\IP^n[k]$-object in $\D(X)$. This means that the following two conditions are satisfied
\begin{enumerate}[label=(C\arabic*)]
 \item\label{C1} The canonical line bundle $\omega_X$ is trivial,
 \item\label{C2} There is an isomorphism of $\IC$-algebras $\Ho^*(\reg_X)\cong \IC[x]/x^{n+1}$ with $\deg x = k$.
\end{enumerate}
\end{definition}
\begin{remark}
If $X$ has a $\IP^n[k]$-unit, we have $\dim X=n\cdot k$. This follows by Serre duality.
\end{remark}
\begin{remark}\label{CYHK}
For $n=1$, compact K\"ahler manifolds with $\IP^1[k]$-units are exactly the strict Calabi--Yau manifolds.
For $k=2$, compact K\"ahler manifolds with $\IP^n[2]$-units are exactly the compact hyperk\"ahler manifolds; see Observations \ref{CYobs} and \ref{HKobs} and Remark \ref{Pobjremark}.   
\end{remark}
\begin{remark}\label{keven}
If $n\ge 2$, the number $k$ must be even. The reason is that the algebra $\Ho^*(\reg_X)$ is graded symmetric. Hence, every $x\in \Ho^k(\reg_X)$ with $k$ odd satisfies $x^2=0$.
\end{remark}
Since, in the following, we usually consider the case that $k>2$, we will speak about \textit{varieties with $\IP^n[k]$-units}; compare Lemma \ref{algebraiclemma}.
\subsection{Non-examples}\label{nonexamplesection}
In order to get a better understanding of the notion of varieties with $\IP^n[k]$-units, it might be instructive to start with some examples which satisfy some of the conditions but fail to satisfy others.
\subsubsection{Products of Calabi--Yau varieties}
Let $Z\in \CY_8$ and $Z'\in \CY_4$, and set $X:=Z\times Z'$. Then $\omega_X$ is trivial and by the K\"unneth formula
\[
 \Ho^*(\reg_X)\cong \IC[0]\oplus \IC[-4]\oplus \IC[-8]\oplus \IC[-12]\,.
\]
Hence, as a graded vector space, $\Ho^*(\reg_X)$ has the right shape for a $\IP^3[4]$-unit. As an isomorphism of graded algebras, however, the K\"unneth formula gives 
\[
 \Ho^*(\reg_X)\cong \IC[z]/z^2\otimes \IC[z']/z'^2\cong \IC[z, z']/(z^2, z'^2)\quad,\quad \deg z=8\,,\, \deg z'=4\,.
\]
This means that, as a $\IC$-algebra, $\Ho^*(\reg_X)$ is not generated in degree $4$ so that $\reg_X$ is not a $\IP^3[4]$-object.
\subsubsection{Hilbert schemes of points on Calabi--Yau varieties}
For every smooth projective variety $X$ and $n=2,3$, the Hilbert schemes $X^{[n]}$ of $n$ points on $X$ are smooth and projective of dimension $n\cdot \dim X$. If $\dim X\ge 3$ and $n\ge 4$, the Hilbert scheme $X^{[n]}$ is not smooth. 

Let now $X$ be a Calabi--Yau variety of even dimension $k$ and $n=2$ or $n=3$. Then there is an isomorphism of algebras $\Ho^*(\reg_{X^{[n]}})\cong \IC[x]/(x^{n+1})$ with $\deg x=k$. The reason is that $X^{[n]}$ is a resolution of the singularities of the symmetric quotient variety $X^n/\sym_n$, which has rational singularities, by means of the Hilbert--Chow morphism $X^{[n]}\to X^n/\sym_n$. 
For $k=2$, the Hilbert scheme of points on a K3 surface is one of the few known examples of a compact hyperk\"ahler manifold which means that $X^{[n]}$ has a $\IP^n[2]$-unit for $X\in \Kthree$.
For $\dim X=k> 2$, however, the canonical bundle $\omega_{X^{[n]}}$ is not trivial as this resolution is not crepant.

In contrast, the symmetric quotient stack $[X^n/\sym_n]$ has a trivial canonical bundle for $\dim X=k$ an arbitrary even number, and is, in fact, a stack with $\IP^n[k]$-unit; see Section \ref{Pnkstacks} for some further details.

\subsection{Main construction method}\label{mainconstruction} 
In this section, given strict Enriques varieties of index $n+1$ we construct a series of varieties with $\IP^n[2k]$-varieties. In other words, we prove the implication (iii)$\Longrightarrow$(ii) of Theorem \ref{main}.

Let $E_1,\dots,E_k$ be strict Enriques varieties of index $n+1$. 
We do not assume that the $E_i$ are non-isomorphic.
For the time being, there are known examples of such $E_i$ for $n=1,2,3$; see Theorem \ref{Enriquesexistence}. We set $F:=E_1\times \dots \times E_k$.   

\begin{theorem}\label{existencethm}
The canonical cover $X:=\widetilde F$ of $F$ has a $\IP^n[2k]$-unit.
\end{theorem}
\begin{proof}
By definition of the canonical cover, $\omega_X$ is trivial. Hence, Condition \ref{C1} of Definition \ref{Pnkdef} is satisfied. It is left to show that $\Ho^*(\reg_X)\cong \IC[x]/x^{n+1}$ with $\deg x= 2k$.
Let $\pi\colon X\to F$ be the \'etale cover with $\pi_*\reg_X\cong \reg_F\oplus \omega_F^{-1}\oplus \dots\oplus \omega_F^{-n}$. Note that $\omega_F\cong \omega_{E_1}\boxtimes\dots \boxtimes \omega_{E_k}$.
By the K\"unneth formula together with Lemma \ref{Enriquescanonical}, we get 
\begin{align*}
 \Ho^*(\reg_X)&\cong \Ho^*(\reg_F)\oplus \Ho^*(\omega_F^{-1})\oplus \dots \oplus \Ho^*(\omega_F^{-n})\\
&\cong \bigl(\otimes_{i=1}^k \Ho^*(\reg_{E_i})\bigr)\oplus \bigl(\otimes_{i=1}^k \Ho^*(\omega^{-1}_{E_i})\bigr)\oplus\dots\oplus \bigl(\otimes_{i=1}^k \Ho^*(\omega^{-n}_{E_i})\bigr)\\
&\cong \IC\oplus \IC\cdot y_1\cdots y_k\oplus \dots \oplus \IC\cdot y_1^n\cdots y_k^n\\
&\cong \IC[x]/x^{n+1}
\end{align*}
where $x:=y_1\cdots y_n$ is of degree $2k$.
\end{proof}
\begin{remark}\label{quotientdecription}
Let $f_i\in \Aut(Y_i)$ be a generator of the group of deck transformations of the cover $Y_i\to E_i$. In other words,  $E_i=Y_i/\langle f_i\rangle$. Then we can describe $X$ alternatively as $X=(Y_1\times \dots \times Y_k)/G$ where 
\[
\mu_{n+1}^{k-1}\cong G=\bigl\{f_1^{a_1}\times \dots \times f_k^{a_k}\mid a_1+\dots +a_k\equiv 0 \mod n+1   \bigr\}\subset \Aut(Y_1\times \dots\times Y_n)\,.
\]
\end{remark}
\begin{remark}
In the case $n=1$, one can replace the $Y_i\in \Kthree$ by  strict Calabi--Yau varieties $Z_i$ of dimension $\dim Z_i=d_i$ together with fixed point free involutions $f_i\in \Aut(Z_i)$. Then the same construction gives a variety $X$ with $\IP^1[d_1+\dots +d_k]$-unit, i.e.\ a strict Calabi--Yau variety of dimension $\dim X=d_1+\dots +d_k$. This coincides with a construction of Calabi--Yau varieties by Cynk and Hulek \cite{CH}.  
\end{remark}

\begin{remark}\label{stillfree}
The construction still works if we replace one of the strict Enriques varieties $E_i$ by an Enriques stack. The reason is that the group $G$ still acts freely on $Y_1\times \dots\times Y_k$, even if one of the $f_i$ has fixed points; see Lemma \ref{freeautos}. 
\end{remark}

\section{Structure of varieties with $\IP^n[k]$-units}
\subsection{General properties}\label{generalfundamental}
As mentioned in Remark \ref{CYHK}, varieties with $\IP^1[k]$-units are exactly the strict Calabi--Yau varieties (not necessarily simply connected) and manifolds with $\IP^n[2]$-units are exactly the compact hyperk\"ahler manifolds. Form now on, we will concentrate on the other cases, i.e. we assume that $n> 2$ and $k>2$. By Remark
\ref{keven}, this means that $k$ is even.
\begin{lemma}\label{unicoverPnk}
 Let $X$ be a variety with a $\IP^n[k]$-unit. Then there is an \'etale cover $X'\to X$ of the form $X'\cong \prod_i Y_i\times\prod_j Z_j$ with $Y_i\in \HK$ and $Z_i\in \CY$ of even dimension.
\end{lemma}
\begin{proof}
Let $X'=T\times Y_i\times\prod_j Z_j$ be a Beauville--Bogomolov cover of $X$ as in Convention \ref{BBcoverconv}. The proof is the same as the first part of the proof of Proposition \ref{HKchar}: We have $\chi(\reg_X)=n+1\neq 0$, hence $\chi(\reg_{X'})\neq 0$. It follows that there cannot be a torus or an odd dimensional Calabi--Yau factor occurring in the decomposition on $X'$.
\end{proof}
Since $\Ho^*(\reg_X)=\IC[x]/x^{n+1}$ with $\deg x=k\ge 4$, we see by the K\"unneth formula that $X'\to X$ cannot be an isomorphism; compare (\ref{productcoh}). 
\begin{cor}
$X'=\widehat X$ is the universal cover of $X$. Hence, $X=X'/G$ for some finite non-trivial subgroup $\pi_1(X)\cong G\subset \Aut(X')$. In particular, a variety with $\IP^n[k]$-unit, for $n>1$, $k\ge 4$ is never simply connected but its fundamental group is always finite.  
\end{cor}
\subsection{The case $k=4$}\label{k=4Sec}
Now, we focus on the case $k=4$ where we can determine the decomposition of the universal cover concretely.

\begin{prop}\label{k=4prop}
 Let $n\ge 3$, and let $X$ be a variety with $\IP^n[4]$-unit. Then the universal cover $\widehat X$ is a product of two hyperk\"ahler varieties of dimension $2n$.
\end{prop}
We divide the proof of this statement into several lemmas. 
So, in the following, let $X$ be a variety with a $\IP^n[4]$-unit where $n\ge 3$.
\begin{lemma}
The universal cover $\widehat X$ of $X$ is a product of compact hyperk\"ahler manifolds.  
\end{lemma}
\begin{proof}
By Lemma \ref{unicoverPnk}, we have $\widehat X\cong \prod_i Y_i\times \prod_j Z_j$ with $Y_i\in \HK_{2d_i}$ and $Z_j\in \CY_{e_j}$ with $e_i\ge 4$ even.
Let $\pi_1(X)\cong G\subset \Aut(\widehat X)$ such that $X=\widehat X/G$. 
Analogously to the proof of Proposition \ref{HKchar}, we see that there is an $x\in\Ho^4(\reg_{\widehat X})^G\cong \Ho^4(\reg_X)$ such that $x^n$ is a non-zero multiple of the generator $\prod_i y_i^{d_i}\cdot \prod_j z_j$ of $\Ho^{4n}(\reg_{\widehat X})$. In particular, all the $z_j$ have to occur in the expression of $x\in\Ho^4(\reg_{\widehat X})$ in terms of the K\"unneth formula. Hence, $e_j=4$ for all $j$. We get
\begin{align}\label{tformula}
 x=\sum_j z_j + \quad\text{terms involving the $y_i$} 
\end{align}
where we absorb possible non-zero coefficients in the choice of the generators $z_j$ of $\Ho^4(\reg_{Z_j})$. Both summands of (\ref{tformula}) are $G$-invariant. This follows by the $G$-invariance of $x$ together with Corollary \ref{inducedaction}. Hence, one of the two summands must vanish. Consequently, $\widehat X$ either has no Calabi--Yau or no hyperk\"ahler factors, i.e. $\widehat X=\prod Y_i$ or $\widehat X=\prod Z_j$. 

Let us assume for a contradiction that the latter is the case. We have $e_j=\dim Z_j= 4$ for all $j$. Since $\dim \widehat X=\dim X=4n$, there must be $n$ factors $Z_j\in \CY_4$ of $\widehat X$. Hence, $\chi(\reg_{\widehat X})=2^n$. 
By Lemma \ref{EulerGlemma}, we have 
\begin{align}\label{EulerG}
\chi(\reg_{\widehat X})=\chi(\reg_X)\cdot \ord (G)\,. 
\end{align}
Hence, $\chi(\reg_X)=n+1 \mid 2^n$. 
Furthermore, $G$ must act transitively on $\{z_1, \dots, z_n\}$. Otherwise, there would be $G$-invariant summands of $x=\sum z_j$ contradicting the assumption that $h^4(\reg_X)=1$. Hence, $n\mid \ord G \mid 2^n$ which, for $n\ge 2$, is not consistent with $n+1\mid 2^n$.
\end{proof}

Hence, we have $\widehat X\cong \prod_{i\in I} Y_i$ with $Y_i\in \HK_{2d_i}$ for some finite index set $I$ and there is a $G$-invariant 
\begin{align}\label{tHKcase}
 0\neq x=\sum_i c_{ii} y_i^2 + \sum_{i\neq j} c_{ij}y_i y_j\in \Ho^4(\reg_{\widehat X}) \quad,\quad c_{ij}\in \IC\,.
\end{align}
Again by Corollary \ref{inducedaction}, both summands in (\ref{tHKcase}) are $G$-invariant so that one of them must be zero.

\begin{lemma}\label{yiyj}
There is a non-zero $G$-invariant $x\in\Ho^4(\reg_{\widehat X})$ of the form $x=\sum_{i\neq j}c_{ij}\cdot y_iy_j$.
\end{lemma}
\begin{proof}
Let us assume for a contradiction that we are in the case that $x=\sum_i y_i^2$ where we hide the coefficients $c_{ii}$ in the choice of the $y_i$. By the same arguments as above, $G$ must act transitively on the set of $y_i$. Hence, by Corollary \ref{inducedaction}, we have $\widehat X=Y^\ell$, $Y\in \HK_{2d}$ with $d\ell=2n$. We must have $\ell\ge 2$ by Corollary \ref{quotientcanonical}. 
Then 
\[
 x^2=\sum_i y_i^4 +2\sum_{i\neq j} y_i^2y_j^2 \,\in \Ho^8(\reg_{\widehat X})^G
\]
and both summands are $G$-invariant. Hence, one of them must be zero and the only possibility for that to happen is that $d<4$. Since $x^n$ is a scalar multiple of the generator $y_1^dy_2^d\cdots y_\ell^d$ of $\Ho^{4n}(\reg_{\widehat X})$, we must have $d=2$. Hence, $\ell=n$ and $\chi(\reg_{\widehat X})=3^n$. By (\ref{EulerG}) and the fact that $G$ acts transitively on $\{y_1,\dots,y_n\}$, we get the contradiction $n\mid 3^n$ and $n+1\mid 3^n$.
\end{proof}

\begin{lemma}
 We have $|I|=2$ which means that $\widehat X\cong Y\times Y'$ with $Y,Y'\in \HK$.
\end{lemma}
\begin{proof}
Let $0\neq x=\sum_{i\neq j}c_{ij} y_iy_j\in \Ho^4(\reg_{\widehat X})^G$ with $c_{ij}\in \IC$, some of which might be zero, as in Lemma \ref{yiyj}. 
As already noted above, we have $|I|\ge 2$ by Corollary \ref{quotientcanonical}.
Let us assume that $|I|\ge 3$. This assumption will be divided into several subcases, each of which leads to a contradiction.
We have 
\begin{align}\label{tsquare}
 x^2=\sum_{i\neq j} c_{ij}^2\cdot y_i^2y_j^2 +\sum_{h\neq i\neq j}c_{hi}c_{ij}\cdot y_hy_i^2y_j+\sum_{g\neq h\neq i\neq j} \hat c_{ghij}\cdot y_gy_hy_iy_j \,,\quad\hat c_{ghij}=c_{gh}c_{ij}+\dots
\,.\end{align}
All three summands are $G$-invariant by Corollary \ref{inducedaction}, hence two of them must be zero.
For one of the first two summands of (\ref{tsquare}) to be zero, the square of some $y_i$ must be zero, i.e. some $Y_{i_0}$ must be a K3 surface. Write the index set $I$ of the decomposition $\widehat X=\prod_{i\in I} Y_i$ as $I=N\uplus M$ where $N=G\cdot i_0$ is the orbit of $i_0$. Here we consider the $G$-action on $I$ given by the permutation part of the autoequivalences in $G\subset \Aut(\widehat X)$; see Lemma \ref{productauto}. With this notation, $Y_j\cong Y_{i_0}\in \Kthree$ for $j\in N$.

Let us first consider the case that $G$ acts transitively on the factors of the decomposition of $\widehat X$, i.e. $I=N$. Then, by dimension reasons, $|I|=2n$. In other words, $\widehat X\cong Y^{2n}$ with $Y\in \Kthree$. Hence, $\chi(\reg_{\widehat X})=2^{2n}$. By (\ref{EulerG}) we get the contradiction $2n\mid 2^{2n}$ and $n+1\mid 2^{2n}$. 

In the case that $M\neq 0$, all the non-zero coefficients $c_{ij}$ in the $G$-invariant $x=\sum_{i\neq j} c_{ij}y_iy_j$ must be of the form $i\in N$ and $j\in M$ (or the other way around). Indeed, otherwise we would have $G$-invariant proper summands of $x$ in contradiction to the assumption $\Ho^4(\reg_{\widehat X})^G=\langle x\rangle$. Furthermore, for all $i\in N$ there must be a non-zero $c_{ii'}$ and for all $j'\in M$ there must be a non-zero $c_{jj'}$ since $x^n$ is a non-zero multiple of the generator $\prod_{i\in N} y_i \cdot \prod_{j\in M}y_j^{d_j}$ of $\Ho^{4n}(\reg_{\widehat X})$. Hence, to avoid proper $G$-invariant summands of $x$, the group $G$ must also act transitively on $M$. It follows that 
$\widehat X\cong Y^\ell\times (Y')^{\ell'}$ 
where $\ell=|N|$, $\ell'=|M|$, $Y\in \Kthree$, and $Y'\in \HK_{2d'}$ for some $d'$. Now, $x^n$ is a non-zero multiple of 
\[\prod_{i=1}^\ell y_i\cdot \prod_{j=1}^{\ell'}(y_j')^{d'}\in \Ho^{4n}(\reg_{\widehat X})\,.\] 
Since all the non-zero summands of $x$ are of the form $c_{ij} y_iy_j'$, we get that $\ell=n=\ell'\cdot d'$. In particular,
\begin{align}\label{case3}\widehat X\cong Y^n\times (Y')^{\ell'}\,.\end{align}

First, we consider for a contradiction the case that $\ell'=1$, hence $\widehat X=Y^n\times Y'$ with $Y\in \Kthree$ and $Y'\in \HK_{2n}$. Then, by (\ref{EulerG}), we get $\ord G=2^n$. We have (up to coefficients which we avoid by the correct choice of the $y_i$), $x=\sum_{i=1}^n y_iy'$. Accordingly, $x^2=\sum_{i\neq j} y_iy_j(y')^2$. Hence, $G$ acts transitively on $\{y_1,\dots ,y_n\}$ as well as on $\{y_iy_j\mid 1\le i< j\le n\}$. We get the contradiction $n\mid 2^n$ and $\binom n2 \mid 2^n$.

Note that, for this to be a contradiction, we need the assumption $n\ge 3$. Indeed, in Section \ref{n=2}, we will see examples of a variety $X$ with a $\IP^2[4]$-unit whose canonical covers are of the form $\widehat X= Y^2\times Y'$ with $Y\in \Kthree$ and $Y'\in \HK_4$. 

Now, let $\ell'>1$ in (\ref{case3}). Then, we get
\begin{align}\label{t22}
 x^2=\sum_{i\neq j, i'}c_{ii'}c_{ji'}\cdot y_iy_{j}(y'_{i'})^2 + \sum_{i\neq j, i'\neq j'}\tilde c_{iji'j'}\cdot y_iy_jy'_{i'}y'_{j'} \quad,\quad \tilde c_{iji'j'}=c_{ii'}c_{jj'}+c_{ij'}c_{ji'} 
\end{align}
where both summands are $G$-invariant. Hence, in order to avoid linearly independent classes in $\Ho^8(\reg_{\widehat X})^G$, one of them must be zero. 

Let us assume for a contradiction that all the $\tilde c_{iji'j'}$ are zero. Then all the $c_{ii'}$ with $i\in N$ and $i'\in M$ are non-zero. Indeed, as mentioned above, given $i\in N$ and $i'\in M$, there exist $j\in N$ and $j'\in M$ such that $c_{ij'}\neq 0 \neq c_{ji'}$. By $\tilde c_{iji'j'}=0$, it follows that also $c_{ii'}\neq 0\neq c_{jj'}$. Given pairwise distinct $h,i,j\in N$ and $i',j'\in M$ we consider the following term, which is the coefficient of $y_hy_iy_j(y'_{i'})^2y_{j'}$ in $x^3$,  
\begin{align}\label{Cformula}C:&=c_{hi'}c_{ii'}c_{jj'}+c_{hi'}c_{ij'}c_{ji'}+c_{hj'}c_{ii'}c_{ji'}\\
&=c_{hi'}\tilde  c_{iji'j'} + c_{hj'}c_{ii'}c_{ji'}\notag \\
&=c_{ii'}\tilde  c_{hji'j'} + c_{hi'}c_{ij'}c_{ji'}\notag \\
&=c_{ji'}\tilde  c_{hii'j'} + c_{hi'}c_{ii'}c_{jj'}\notag \,.
\end{align}
By the vanishing of the $\tilde c$, we get 
\[
C=c_{hi'}c_{ii'}c_{jj'}=c_{hi'}c_{ij'}c_{ji'}=c_{hj'}c_{ii'}c_{ji'}\,.
\]
By the non-vanishing of all the $c$, we get $C\neq 0$. But, at the same time, by (\ref{Cformula}), we have $3C=C$; a contradiction.

We conclude that the first summand of (\ref{t22}) is zero. This can only happen for $(y_{i'}')^2=0$, hence $Y'\in \Kthree$. Then $\chi(\reg_{\widehat X})=2^{2n}$ and, as before, we get the contradiction that $n\mid 2^{2n}$ and $n+1\mid 2^{2n}$. 
\end{proof}

\begin{proof}[Proof of Proposition \ref{k=4prop}]
By now, we know that $\widehat X=Y\times Y'$ with $Y\in\HK_{2d}$, $Y'\in\HK_{2d'}$, and $x=yy'$. We have $d+d'=2n$. Furthermore, $0\neq x^n=y^n(y')^n$. Hence, $d=n=d'$. 
\end{proof}

\begin{remark}
The proof of Proposition \ref{k=4prop} becomes considerably simpler if one assumes that $n+1$ is a prime number. In this case, it follows directly by Lemma \ref{EulerGlemma} that the universal cover must have a factor $Y\in \HK_{2n}$. Hence, there are much fewer cases one has to deal with. 
\end{remark}
\begin{theorem}\label{4theorem}
Let $n\ge 3$, and let $X$ be a variety with a $\IP^n[4]$-unit. 
\begin{enumerate} 
\item\label{narb} We have $X=(Y\times Y')/G$ with $Y,Y'\in \HK_{2n}$. The group $\pi_1(X)\cong G\subset \Aut(Y\times Y')$ acts without fixed points, and is of the form $G=\langle f \times f'\rangle$ with $f\in \Aut(Y)$ and $f\in \Aut(Y')$ purely symplectic of order $n+1$.
\item\label{nprime} If $n+1=p^\nu$ is a prime power, at least one of the cyclic groups $\langle f\rangle\subset \Aut(Y)$ and  $\langle f'\rangle \subset\Aut(Y')$ acts without fixed points.     
\end{enumerate}
\end{theorem}
Before giving the proof of the theorem, let us restate, for convenience, the special case  of Lemma \ref{freeautos} for automorphisms of products with two factors.
\begin{lemma}\label{freeautostwo}
Let $X$ and $Y$ be manifolds, $g,f\in \Aut(X)$ and $h\in \Aut(Y)$.
\begin{enumerate}
 \item\label{freeautosdiagonal} $g\times h\in \Aut(X\times Y)$ is fixed point free if and only if at least one of $g$ and $h$ is fixed point free.
 \item\label{freeautospermute} Let $\phi:=(f\times g)\circ (1\,\,2)\in \Aut(X^2)$ be given by $(a,b)\mapsto (f(b), g(a))$. Then, $\phi$ is fixed point free if and only if $f\circ g$ and $g\circ f$ are fixed point free. 
\end{enumerate}
\end{lemma}
\begin{proof}[Proof of Theorem \ref{4theorem}]
The fact that $X=(Y\times Y')/G$ with $Y,Y'\in \HK_{2n}$ and $G\cong \pi_1(X)$ is just a reformulation of Proposition \ref{k=4prop}. By the proof of this proposition, we see that $\Ho^*(\reg_{Y\times Y'})^G\cong \Ho^*(\reg_X)$ is generated by $x=yy'$ in degree 4. 

Let us assume for a contradiction that $G$ contains an element which permutes the factors $Y$ and $Y'$, in which case we have $Y=Y'$ by Lemma \ref{productauto}. In other words, there exists an $\phi=(f\times g)\circ (1\,\,2)\in G$ as in Lemma \ref{freeautostwo} \ref{freeautospermute}. Hence, $f\circ g$ is fixed point free. By  
Lemma \ref{HKfreeauto}, the composition $f\circ g$ is non-symplectic, i.e. $\rho_{f\circ g}\neq 1$. But $\rho_{f\circ g}=\rho_f\cdot \rho_g$ so that $\phi$ acts non-trivially on $x=yy'$ in contradiction to the $G$-invariance of $x$. 

Hence, every element of $G$ is of the form $g\times h$ as in Lemma \ref{freeautos} \ref{freeautosdiagonal}. We consider the group homomorphisms 
$\rho_Y\colon G\to \IC^*$ and $\rho_{Y'}\colon G\to \IC^*$. Their images are of the form $\mu_m$ and $\mu_{m'}$ respectively. We must have $m,m'\ge n+1$. Indeed, $y^m$ and $(y')^{m'}$ are $G$-invariant but, for $m\le n$ or $m'\le n$, not contained in the algebra generated by $x=yy'$. Since $|G|=n+1$, assertion \ref{narb} follows.

Let now $n+1=p^\nu$ be a prime power and $G=\langle f\rangle$. Let us assume for a contradiction that there exist $a,b\in \IN$ with $n+1=p^\nu \nmid a,b$ such that $f^a$ and $(f')^b$ have fixed points. Note that, in general, if an automorphism $g$ has fixed-points, also all of its powers have fixed points. Furthermore, for two elements $a,b\in \IZ/(p^\nu)$ we have $a\in \langle b\rangle$ or $b\in \langle a\rangle$. Hence, $(f\times f')^a$ or $(f\times f')^b$ has fixed points in contradiction to part \ref{narb}.
\end{proof}
This proves the implication (i)$\Longrightarrow$(iii) of Theorem \ref{main}.
Indeed, for $n+1$ a prime power, the above Theorem says that $Y/\langle f\rangle$ or $Y'/\langle f'\rangle$ is a strict Enriques variety; see Proposition \ref{classEnriques} \ref{Enriquesquotient}.   
 Note that Theorem 
\ref{4theorem} above does not hold for $n=2$; see Section \ref{n=2}. However, both conditions (i) and (iii) of Theorem \ref{main} hold true for $n=2$; see Theorem \ref{Enriquesexistence} and Corollary \ref{existencecor}. 
\begin{remark}
The proof of part \ref{nprime} of Theorem \ref{4theorem} does not work if $n+1$ is not a prime power. For example, if $n+1=6$, one could obtain a variety with $\IP^5[4]$-unit as a quotient $X=(Y\times Y')/\langle f\times g\rangle$ with $Y,Y'\in \HK_{10}$ such that $f$ and $g$ are purely non-symplectic of order 6, and $f$, $f^2$, $f^4$, $f^5$, $g$, $g^3$, $g^5$ are fixed point free but $f^3$, $g^2$, and $g^4$ are not. The author does not know whether hyperk\"ahler manifolds together with these kinds of automorphisms exist. 
\end{remark}


\section{Further remarks}\label{furtherremarks}

\subsection{Further constructions using strict Enriques varieties}\label{furtherconst1}
Given strict Enriques varieties of index $n+1$, there are, for $k\ge 6$, 
further constructions of varieties with $\IP^n[k]$-units besides the one 
of Section \ref{mainconstruction}.
Let $Y\in \HK_{2n}$ together with an $f\in \Aut(Y)$ purely symplectic of order $n+1$ 
such that $\langle f\rangle$ acts without fixed points, i.e\ the 
quotient $E=Y/\langle f\rangle$ is a strict Enriques variety. We 
consider the $(n+1)$-cycle $\sigma:=(1\,\,2 \cdots n+1)\in\sym_{n+1}$ and 
the subgroup $G(Y)\subset \Aut(Y^{n+1})$ given by
\[
  G(Y):=\bigl\{(f^{a_1}\times \dots\times f^{a_{n+1}})\circ \sigma^a\mid 
a_1+\dots +a_{n+1}\equiv a\mod n+1  \bigr\}\,.
\]
Every non-trivial automorphism of $G$ acts fixed point free on $Y^{n+1}$ 
by Lemma \ref{freeautos}. 
There is the surjective homomorphism
\[
 G(Y)\to \IZ/(n+1)\IZ \quad, \quad (f^{a_1}\times \dots\times f^{a_{n+1}})\circ \sigma^a\mapsto a\mod n+1
\]
and we denote the fibres of this homomorphism by $G_a(Y)$.

Now, consider further $Z_1,\dots, Z_k\in\HK_{2n}$ together 
with purely non-symplectic $g_i\in \Aut(Z_i)$ of order $n+1$ such that $\langle g_i\rangle$ acts freely and 
\begin{align}\label{rhoformula}\rho_{Z_i,g_i}=\rho_{Y,f}\quad \text{for all $i=1,\dots,k$.} 
\end{align}
The equality (\ref{rhoformula}) can be achieved as soon as we have any purely non-symplectic automorphisms $g_i\in \Aut(Z_i)$ of order $n+1$ by replacing the $g_i$ by appropriate powers $g_i^{\beta_i}$ with $\gcd(\beta_i,n+1)=1$.
We consider the subgroup $G(Y;Z_1,\dots, 
Z_k)\subset \Aut(Y^{n+1}\times Z_1\times \dots\times Z_k)$ given by
\[
  G(Y;Z_1,\dots, Z_k):=\bigl\{F\times g_1^{b_1}\times\dots\times g_k^{b_k}\mid F\in G_a(Y)\,,\, 
a+b_1+\dots +b_k\equiv 0\mod n+1  \bigr\}\,.
\]
\begin{prop}\label{secondconstruction}
  The quotient $X:=(Y^{n+1}\times Z_1\times \dots \times 
Z_k)/G(Y;Z_1,\dots, Z_k)$ is a smooth projective variety with 
$\IP^n[2(n+1+k)]$-unit.
\end{prop}
\begin{proof}
One can check using  Lemma \ref{freeautos} that the group $G:=G(Y;Z_1,\dots,Z_k)$ acts freely on $X':=Y^{n+1}\times Z_1\times\dots \times 
Z_k$.  
Hence, $X$ is indeed smooth. 

By the defining property of the elements of $G(Y; Z_1,\dots, Z_k)$ together with (\ref{rhoformula}), we see that $x:=y_1y_2\cdots y_{n+1}z_1z_2\cdots z_k$ is $G$-invariant. Hence, as $x^i\neq 0$ for $0\le i\le n$, we get the inclusion 
\begin{align}\label{inclusioninvariants}\IC[x]/x^{n+1}\subset \Ho^*(\reg_{X'})^G\cong \Ho^*(\reg_X)\quad,\quad \deg x= 2(n+1+k)\,.\end{align} 
Also, $\ord G(Z_1,\dots,Z_k)=(n+1)^{n+1+k-1}$. By Lemma \ref{EulerGlemma}, we get $\chi(\reg_{X})=n+1$ so that the inclusion (\ref{inclusioninvariants}) must be an equality which is \ref{C2}. 
Finally, the canonical bundle $\omega_X$ is trivial since $G$ acts trivially on $\langle x^n\rangle=\Ho^{\dim X'}(\reg_{X'})\cong \Ho^0(\omega_{X'})$. 
\end{proof}
\begin{remark}
For $n\ge 2$, the group $G(Y;Z_1,\dots,Z_k)$ is not abelian. Since $X'\to X$ is the universal cover, we see that, for $k\ge 4$, there are examples of varieties with $\IP^n[k]$-units which have a non-abelian fundamental group.  
\end{remark}
\begin{remark}
Again, for one $i\in\{1,\dots k\}$ we may drop the assumption that $\langle g_i\rangle$ acts freely; compare Remark \ref{stillfree}. 
\end{remark}

\begin{remark}
 One can further generalise the above construction as follows. Consider hyperk\"ahler manifolds  $Y_1,\dots, Y_m,Z_1,\dots, Z_k\in \HK_{2n}$ together with $f_i\in \Aut(Y_i)$ and $g_j\in \Aut(Z_j)$ purely non-symplectic of order $n+1$ such that the generated cyclic groups act freely. Set $X':=Y_1^{n+1}\times \dots \times Y_m^{n+1}\times Z_1\times\dots Z_k$ and consider $G:=G(Y_1,\dots, Y_m;Z_1,\dots,Z_k)\subset \Aut(X')$ given by
\[
G=\bigl\{F_1\times \dots \times F_m\times g_1^{b_1}\times\dots\times g_k^{b_k}\mid F_i\in G_{a_i}(Y)\,,\, 
a_1+\dots +a_m+b_1+\dots +b_k\equiv 0\mod n+1  \bigr\}\,.
\]
Then, $X:=X'/G$ has a $\IP^n[2(m(n+1)+k)]$-unit.
\end{remark}

\begin{remark}
 In the case $n=1$, one may replace the K3 surfaces $Y_i$ and $Z_j$ by strict Calabi--Yau varieties of arbitrary dimensions. Still, the quotient $X$ will be a strict Calabi--Yau variety.
\end{remark}

\subsection{A construction not involving strict Enriques varieties}\label{n=2}
As mentioned in Section \ref{k=4Sec}, there is a variety $X$ with $\IP^2[4]$-unit whose universal cover 
$\widehat X$ is not a product of two hyperk\"ahler varieties of dimension $4$. This shows that the assumption $n\ge 3$ in Proposition \ref{k=4prop} is really necessary. 

For the construction, let $Z$ be a strict Calabi--Yau variety of dimension $\dim Z=e$ together with a fixed point free involution $\iota \in \Aut(Z)$. Necessarily, $\rho_{Z,\iota}=-1$; see Lemma \ref{HKfreeauto}. 
Furthermore, let $Y\in \HK_4$ together with a purely non-symplectic $f\in \Aut(Y)$ of order $4$ with $\rho_{Y,f}=\sqrt{-1}$.
Note that $g$ must have fixed points on $Y$.
Such pairs $(Y,f)$ exist. Take a K3 surface $S$ (an abelian surface $A$) together with a purely non-symplectic automorphism of order $4$ and $Y=S^{[2]}$ ($Y=K_2A$) together with the induced automorphism.

Now, consider $G(Z)\subset \Aut(Z^2)$ as in the previous section. It is a cyclic group of order $4$ with generator $g=(\iota\times \id)\circ (1\,\,2)$. Set $X'=Y\times Z^2$ and $G:=\langle f\times g\rangle \subset \Aut(X')$. The group $G$ acts freely, since $G(Z)$ does; see Lemma \ref{freeautostwo}. One can check that
$x=yz_1 + \sqrt{-1}\cdot y z_2\in \Ho^{2+e}(\reg_{X'})$ is $G$-invariant. By the same argument as in the proof of Proposition \ref{secondconstruction}, we conclude that $X$ has a $\IP^2[2+e]$-unit. In particular, in the case that $Z\in \Kthree$, we get a variety with $\IP^2[4]$-unit.

\subsection{Possible construction for $k=6$}\label{furtherconst3}

In contrast to the case $k=4$ and $n+1$ a prime power (see Theorem \ref{main}), there might be a variety with $\IP^n[6]$-unit even if there is no Enriques variety of index $n+1$ but one of index $2n+1$. 
Of course, since there are at the moment only known examples of strict Enriques varieties of index $2$, $3$, and $4$, this is only hypothetical. 

Indeed, let $Y\in \HK_{4n}$ together with subgroup $\langle f\rangle\subset \Aut (Y)$ acting freely, where $f$ is purely non-symplectic of order $2n+1$, and let $Y'\in \HK_{2n}$ together with $f'\in \Aut(Y')$ non-symplectic of order $n+1$ with $\rho_{Y, f}=\rho_{Y',f'}^{-1}$. Necessarily, $f'$ has fixed points; see Lemma \ref{HKfreeauto}. Then $G=\langle f\times f'^2\rangle$ acts freely on $Y$ and $x=y^2\cdot y'$ is $G$-invariant. It follows that $X=(Y\times Y')/G$ has a $\IP^n[6]$-unit.

\subsection{Stacks with $\IP^n[k]$-units}\label{Pnkstacks}
Let $\cX$ be a smooth projective stack. In complete analogy to the case of varieties, we say that $\cX$ \textit{has a $\IP^n[k]$-unit} if $\reg_\cX\in \D(\cX)$ is a $\IP^n[k]$-object. Again, this means that:
\begin{enumerate}[label=(C\arabic*')]
 \item\label{C1'} The canonical line bundle $\omega_\cX$ is trivial,
 \item\label{C2'} There is an isomorphism of $\IC$-algebras $\Ho^*(\reg_{\cX})\cong \IC[x]/x^{n+1}$ with $\deg x = k$.
\end{enumerate}
In contrast to the case of varieties, it is very easy to construct stacks with $\IP^n[k]$-units. 

Let $Z\in \CY_k$ with $k$ even. Then, the symmetric group $\sym_n$ acts on $Z^n$ by permutation of the factors and we call the associated quotient stack $\cX=[Z^n/\sym_n]$ the \textit{symmetric quotient stack}. Then, as $k=\dim Z$ is even, the canonical bundle of $\cX$ is trivial; see \cite[Sect.\ 5.4]{KSosequiequi}. Condition \ref{C2'} follows by the K\"unneth formula  
\[
 \Ho^*(\reg_\cX)\cong \Ho^*(\reg_{Z^n})^{\sym_n}\cong (\Ho^*(\reg_Z)^{\otimes n})^{\sym_n}\cong S^n(\Ho^*(\reg_Z))\,.
\]

There are also plenty of other examples of stacks with $\IP^n[k]$-units. Let $S\in \Kthree$ with $\iota\in S$ a non-symplectic involution and $\iota^{[n]}\in \Aut(S^{[n]})$ the induced automorphism on the Hilbert scheme of $n$ points on $S$. Then, for $n$ even, the associated quotient stack $[X^{[n]}/\iota^{[n]}]$ has a $\IP^{n/2}[4]$-unit. In contrast, if $n$ is odd and $\iota$ fixed point free, the quotient $X^{[n]}/\iota^{[n]}$ is an OS Enriques variety; see \cite[Prop.\ 4.1]{OSEnriques}.

Also, all the constructions of the earlier sections lead to stacks with $\IP^n[k]$-units if we replace the strict Enriques varieties by strict Enriques stacks.
\subsection{Derived invariance of strict Enriques varieties}\label{derivedsection}
In \cite{Abuunit}, Abuaf conjectured that the homological unit is a derived invariant of smooth projective varieties. This means that for two varieties $X_1$, $X_2$ with $\D(X_1)\cong \D(X_2)$ we should have an isomorphisms of $\IC$-algebras $\Ho^*(\reg_{X_1})\cong \Ho^*(\reg_{X_2})$. 

In regard to this conjecture, one would like to proof that the class of varieties with $\IP^n[k]$-units is stable under derived equivalence. This is true for $k=2$: In \cite{HuyNW},  it is shown that the class of compact hyperk\"ahler varieties is stable under derived equivalence. However, the methods of the proof do not seem to generalise to higher $k$. At least, we can use the result of \cite{HuyNW} in order to show that the class of strict Enrqiues varieties is derived stable.    

\begin{lemma}
 Let $E_1$ be a strict Enriques variety of index $n+1$ and $E_2$ a Fourier--Mukai partner of $E_2$, i.e.\ $E_2$ is a smooth projective variety with $\D(E_1)\cong \D(E_2)$. Then $E_2$ is also a strict Enriques variety of the same index $n+1$. 
\end{lemma}
\begin{proof}
By Proposition \ref{classEnriques}, condition \ref{S1} of a strict Enriques variety of index $n+1$ can be replaced by the condition $\dim E_1=2n$. The dimension of a variety and the order of its canonical bundle are derived invariants; see e.g.\ \cite[Prop.\ 4.1]{Huy}. Hence, also $\dim E_2=2n$ and $\ord \omega_{E_2}=n+1$. 

It remains to show that the canonical cover $\widetilde{E_2}$ is again hyperk\"ahler. 
Indeed, the equivalence $\D(E_1)\cong \D(E_2)$ lifts to an equivalence of the canonical covers $\D(\widetilde E_1)\cong \D(\widetilde{E_2})$ and the class of hyperk\"ahler varieties is stable under derived equivalences; see \cite{BM} and \cite{HuyNW}, respectively. 
\end{proof}

The exactly same proof shows that the class of OS Enriques varieties with fixed dimension and index is derived stable.

\subsection{Autoequivalences of varieties with $\IP^n[k]$-unit}\label{autosection}
As mentioned in Remark \ref{inducedautoremark}, every $\IP^n[k]$-object $E\in \D(X)$ induces an autoequivalence, called $\IP$-twist, $P_E\in \Aut(\D(X))$. This can be seen as a special case of \cite[Thm.\ 3]{Add} or as a straight-forward generalisation of \cite[Prop.\ 2.6]{HT}. 
We will describe the twist only in the special case $E=\reg_X$. In particular, we assume that $X$ has a $\IP^n[k]$-unit. Then, by Remark 
\ref{regspecial}, every line bundle $L\in \Pic X$ is a $\IP^n[k]$-object too. However, it suffices to understand the twist $P_X:=P_{\reg_X}$ as we have $P_L=M_{L}P_X M_L^{-1}$ where $M_L=(\_)\otimes L$ is the autoequivalence given by tensor product with $L$; see \cite[Lem.\ 2.4]{Kru3}.

The $\IP$-twist along $\reg_X$ is constructed as the Fourier--Mukai transform $P_X:=\FM_{\cQ}\colon \D(X)\to \D(X)$ where 
\[
\mathcal Q=\cone\bigl(\cone(\reg_{X\times X}\xrightarrow{x\boxtimes \id-\id\boxtimes x} \reg_{X\times X})\xrightarrow{r}\reg_{\Delta} \bigr)\in \D(X\times X)\,.
\]
Here, $x$ is a generator of $\Ho^k(\reg_X)\cong \Hom(\reg_X[-k], \reg_X)$ and $r\colon \reg_{X\times X}\to\reg_\Delta$ is the restriction of sections to the diagonal. The double cone makes sense, since $r\circ (x\boxtimes \id-\id \boxtimes x)=0$; see \cite[Sect.\ 2]{HT} for details. 
On the level of objects $F\in \D(X)$, the twist $P_X$ is given by 
\begin{align}\label{Pdescription}
P_X(F) =\cone\Bigl(\cone\bigl(\Ho^*(F)\otimes \reg_{X}[-k]\to \Ho^*(F)\otimes \reg_{X}\bigr)\to F \Bigr)\,.
\end{align}
We summarise the main properties of the twist $P_X$ in the following 
\begin{prop}\label{twistprop}
 The $\IP$-twist $P_X\colon \D(X)\to \D(X)$ is an autoequivalence with the properties
 \begin{enumerate}
  \item $P_X(\reg_X)=\reg_X[-k(n+1)+2]$,
  \item $P_X(F)=F$ for $F\in \reg_X^\perp=\{F\in \D(X)\mid \Hom^*(\reg_X, F)=0\}$,
\item\label{phiitem} Let $\Phi\in \Aut(\D(X))$ with $\Phi(\reg_X)=\reg_X[m]$ for some $m\in \IZ$. Then the autoequivalences $\Phi$ and $P_X$ commute.
 \end{enumerate}
\end{prop}
\begin{proof}
 For the first two properties, see \cite[Sect.\ 2]{HT} or \cite[Sect.\ 3.4\&3.5]{Add}. Part \ref{phiitem} follows from \cite[Lem.\ 2.4]{Kru3}.
\end{proof}

\begin{lemma}\label{complement}
Let $X$ be a variety with $\IP^n[k]$-unit with $k\ge 2$ (not an elliptic curve). Let $Z_1, Z_2\subset X$ be two disjoint closed subvarieties and set \[F:=R\sHom \bigl(P_X(\reg_{Z_1}), P_X(\reg_{Z_2})\bigr)\in \D(X)\,.\] Then $\Hom^*(\reg_X,F)=\Ho^*(F)=0$ and $F\neq 0$.
In particular, the orthogonal complement of $\reg_X$ is non-trivial. 
\end{lemma}
\begin{proof}
Clearly, $\Hom^*(\reg_{Z_1}, \reg_{Z_2})=0$. Using the fact that the equivalence $P_X$ is, in particular, fully faithful and standard compatibilities between derived functors, we get
\[
 0=\Hom^*\bigl(P_X(\reg_{Z_1}),P_X(\reg_{Z_2})\bigr)= \Hom^*\bigl(\reg_X,R\sHom \bigl(P_X(\reg_{Z_1}), P_X(\reg_{Z_2}\bigr)\bigr)\,.
\]
It is left to show that $F:=R\sHom \bigl(P_X(\reg_{Z_1}), P_X(\reg_{Z_2})\bigr)\neq 0$.
We denote by $\alpha_i$ the top non-zero degree of $\Ho^*(\reg_{Z_i})$ for $i=1,2$. Let $V:=X\setminus (Z_1\cup Z_2)$. Then by (\ref{Pdescription}), the cohomology of $P_X(\reg_{Z_i})$ is concentrated in degrees between $-1$ and $\alpha_i+k-2$ with $\cH^{-1}(P_X(\reg_{Z_i}))_{|V}\cong \reg_V$ and $\cH^{\alpha_i+k-2}(P_X(\reg_{Z_i}))_V\cong \reg_V\otimes \Ho^{\alpha_i}(\reg_{Z_i})$. Hence, the spectral sequence
\[
 E^{p,q}_2=\oplus_i\sExt^p\bigl(\cH^i(P(\reg_{Z_1})),\cH^{i+q}(P(\reg_{Z_1}))\bigr)_{|V}\quad \Longrightarrow \quad E^{p+q}=\cH^{p+q}(F)_{|V}
\]
is concentrated in the quadrant to the upper right of $(0,-\alpha_1-k+1)$. Furthermore, we have $E^{0,-\alpha_1-k+1}_2\cong \reg_V\otimes \Ho^{\alpha_1}(\reg_{Z_1})\neq 0$. Hence $\cH^{-\alpha_1-k+1}(F)\neq 0$.   
\end{proof}
Let now $X$ be obtained from strict Enriques varieties via the construction of Section
\ref{mainconstruction}. This means that $X=(Y_1\times \dots\times Y_k)/G$ with $Y_i\in \HK_{2n}$ and 
\[
 G=\bigl\{ f_1^{a_1}\times \dots\times f_k^{a_k}\mid a_1+\dots +a_k\equiv 0 \mod n+1\bigr\}
\]
where the $f_i\in \Aut(Y_i)$ are purely non-symplectic of order $n+1$. There are the $\IP$-twists $P_{Y_i}:=P_{\reg_{Y_i}}\in \Aut(\D(Y_i))$ whose Fourier--Mukai kernels we denote by $\cQ_i$. These induce autoequivalences $P'_{Y_i}:=\FM_{\cQ'_i}\in \Aut(\D(Y_1\times \dots \times Y_k))$ where \[\cQ'_i=\reg_{\Delta Y_1}\boxtimes\dots\boxtimes \cQ_i\boxtimes \dots\boxtimes \reg_{\Delta Y_k}\in\D\bigl((Y_1\times Y_1)\times \dots (Y_i\times Y_i)\times \dots (Y_k\times Y_k)\bigr)\,.\]   
We have 
\begin{align}\label{boxaction}P'_{Y_i}(F_1\boxtimes \dots \boxtimes F_k)=F_1\boxtimes \dots \boxtimes P_{Y_i}(F_i)\boxtimes \dots \boxtimes F_k\,.\end{align} 
We will use in the following the identification $\D(X)\cong \D_G(X')$ of the derived category of $X$ with the derived category of $G$-linearised coherent sheaves on the cover $X'=Y_1\times \dots\times Y_k$; see e.g.\ \cite[Sect.\ 4]{BKR} or \cite{KSosequiequi} for details. One can check that the $\cQ_i$ are $\langle f_i\rangle$-linearisable, hence the $\cQ'_i$ are $G$-linearisable. It follows that the autoequivalences $P'_{Y_i}$ descend to autoequivalences $\check P_{Y_i}\in \Aut(\D_G(X'))\cong \Aut(\D(X))$; see \cite[Thm.\ 1.1]{KSosequiequi}. One might expect that the composition of the $\check P_{Y_i}$ equals $P_X$ but this is not the case.

\begin{prop}
There is an injective group homomorphism $\IZ^{\oplus k+2}\hookrightarrow \Aut(\D(X))$ given by
\[
  e_{k+1}\mapsto P_X\quad,\quad e_{k+2}\mapsto [1]\quad,\quad e_i\mapsto \check P_{Y_i} \quad\text{for $i=1,\dots,k$.}
\]
\end{prop}
\begin{proof}
Under the equivalence $\D(X)\cong \D_G(X')$, the structure sheaf $\reg_X\in \D(X)$ corresponds to $\reg_{X'}=\reg_{Y_1}\boxtimes \dots\boxtimes \reg_{Y_k}$ equipped with the natural linearisation. By (\ref{boxaction}) and Proposition \ref{twistprop}(1), we get 
\[\check P_{Y_i}(\reg_X)\cong \reg_{Y_1}\boxtimes \dots \boxtimes (\reg_{Y_i}[-2n])\boxtimes \dots \boxtimes \reg_{Y_k}\cong \reg_X[-2n]\,.\]
Hence, by \ref{twistprop}(3), the $\check P_{Y_i}$ commute with $P_X$. By a similar argument, one can see that the $\check P_i$ commute with one another. 
The shift functor $[1]$ commutes with every autoequivalence of the triangulated category $\D(X)$. In summary, we have shown by now that the homomorphism $\IZ^{\oplus k+2}\to \Aut(\D(X))$ is well-defined.

For the injectivity, let us fix for every $i=1,\dots, n$ a $G$-linearisable $F_i\in\reg_{Y_i}^{\perp}$. For example, let $Z_1$ and $Z_2$ in Lemma \ref{complement} be two different $\langle f_i \rangle$-orbits in $Y_i$. Let $a_1,\dots,a_k, b,c\in \IZ$ and set $\Psi:=\check P_{Y_1}^{a_1}\circ \dots \circ \check P_{Y_k}^{a_k}\circ P_X^b[c]$. By plugging various box-products of the $\reg_{Y_i}$ and $F_i$ into $\Psi$ we can show that $\Psi\cong \id$ implies $0=a_1=a_2=\dots=a_k=b=c$; this is very similar to computations done in \cite[Sect.\ 1.4]{Add} or the proof of \cite[Prop.\ 3.18]{KSosEnriques}.
\end{proof}

\begin{remark}
 In the known examples, the $Y_i$ are generalised Kummer varieties; compare Section \ref{Enriquessection}. In these cases, there are many more $\IP$-objects in $\D(Y_i)$ which induce further autoequivalences on $X$; see \cite[Sect.\ 6]{Kru3}.  
\end{remark}

\begin{cor}
Let $X$ be a variety with $\IP^n[4]$-unit for $n\ge 3$. Then, there is an embedding 
$\IZ^4\subset \Aut(\D(X))$.
\end{cor}
\begin{proof}
 By Theorem \ref{4theorem}, we are in the situation of the above proposition. 
\end{proof}


\subsection{Varieties with $\IP^n[k]$-units as moduli spaces}\label{moduli}
In all the constructions presented in this article, we start with hyperk\"ahler manifolds with special autoequivalences, usually with the property that the quotients are strict Enriques varieties. Then the varieties with $\IP^n[k]$-units are constructed as intermediate quotients between the product of the hyperk\"ahler manifolds and the product of the quotients.

As already mentioned in the introduction, it would be very interesting to find ways to construct varieties $X$ with $\IP^n[k]$-units directly. In the case $k=4$, by Proposition \ref{k=4prop}, the universal cover of such an $X$ decomposes into two hyperk\"ahler manifolds. Hence, one could hope to find in this way new examples of Enriques or even hyperk\"ahler varieties.

One could try to find examples of varieties with $\IP^n[k]$ units by looking at moduli spaces of sheaves (or objects) on varieties with trivial canonical bundle of dimension $k$. Indeed all of the examples that we found in this paper can be realised as such moduli spaces. 

For example, let $A$, $B$ be abelian surfaces together with automorphisms $a\in \Aut(A)$ and $b\in \Aut(B)$. We set $Y:=K_2A$, $Z:=K_2B$,     
$f:=K_2a$, $g:=K_2b$ and assume that $Y/\langle f\rangle$ and $Z/\langle g\rangle$ are strict Enriques varieties of index $3$. This implies that $X:=(Y\times Z)/\langle f\times g\rangle$ has a $\IP^2[4]$-unit; see Remark \ref{quotientdecription}. As $Y=K_2A$ and $Z=K_2B$ are moduli spaces of sheaves on $A$ and $B$, respectively, the product $Y\times Z$ is a moduli space of sheaves on $A\times B$. We denote the universal family by $\cF\in \Coh(A\times B\times Y\times Z)$. This descends to a sheaf $\check \cF\in \Coh((A\times B)/\langle a \times b\rangle \times X)$ which is flat over $X$ with pairwise non-isomorphic fibres. One can deduce this from the fact that $\cF$ is $\langle a\times b\times f\times g\rangle$-linearisable; compare \cite[Sect.\ 3]{KSosequiequi}. 
Hence, we can consider $X$ as a moduli space of sheaves on $(A\times B)/\langle a\times b\rangle$ with universal family $\check\cF$.

\bibliographystyle{alpha}
\addcontentsline{toc}{chapter}{References}
\bibliography{references}

\begin{thebibliography}{BNWS11}

\bibitem[Abu15]{Abuunit}
Roland Abuaf.
\newblock Homological units.
\newblock {\em arXiv:1510.01583}, 2015.

\bibitem[Add11]{Add}
Nicolas Addington.
\newblock New derived symmetries of some hyperkaehler varieties.
\newblock {\em To appear in Algebr.\ Geom.,\ arXiv:1112.0487}, 2011.

\bibitem[Bea83]{Beasome}
Arnaud Beauville.
\newblock Some remarks on {K}\"ahler manifolds with {$c_{1}=0$}.
\newblock In {\em Classification of algebraic and analytic manifolds ({K}atata,
  1982)}, volume~39 of {\em Progr. Math.}, pages 1--26. Birkh\"auser Boston,
  Boston, MA, 1983.

\bibitem[BKR01]{BKR}
Tom Bridgeland, Alastair King, and Miles Reid.
\newblock The {M}c{K}ay correspondence as an equivalence of derived categories.
\newblock {\em J. Amer. Math. Soc.}, 14(3):535--554 (electronic), 2001.

\bibitem[BM98]{BM}
Tom Bridgeland and Antony Maciocia.
\newblock Fourier--{M}ukai transforms for quotient varieties.
\newblock {\em arXiv:math/9811101}, 1998.

\bibitem[BNWS11]{BNWS}
Samuel Boissi{\`e}re, Marc~A. Nieper-Wi{\ss}kirchen, and Alessandra Sarti.
\newblock Higher dimensional {E}nriques varieties and automorphisms of
  generalized {K}ummer varieties.
\newblock {\em J. Math. Pures Appl. (9)}, 95(5):553--563, 2011.

\bibitem[CH07]{CH}
S.~Cynk and K.~Hulek.
\newblock Higher-dimensional modular {C}alabi-{Y}au manifolds.
\newblock {\em Canad. Math. Bull.}, 50(4):486--503, 2007.

\bibitem[HNW11]{HuyNW}
Daniel Huybrechts and Marc Nieper-Wisskirchen.
\newblock Remarks on derived equivalences of {R}icci-flat manifolds.
\newblock {\em Math. Z.}, 267(3-4):939--963, 2011.

\bibitem[HT06]{HT}
Daniel Huybrechts and Richard Thomas.
\newblock {$\mathbb P$}-objects and autoequivalences of derived categories.
\newblock {\em Math. Res. Lett.}, 13(1):87--98, 2006.

\bibitem[Huy03]{HuyHKbook}
Daniel Huybrechts.
\newblock Compact hyperk\"ahler manifolds.
\newblock In {\em Calabi-{Y}au manifolds and related geometries
  ({N}ordfjordeid, 2001)}, Universitext, pages 161--225. Springer, Berlin,
  2003.

\bibitem[Huy06]{Huy}
Daniel Huybrechts.
\newblock {\em Fourier-{M}ukai transforms in algebraic geometry}.
\newblock Oxford Mathematical Monographs. The Clarendon Press Oxford University
  Press, Oxford, 2006.

\bibitem[Kru15]{Kru3}
Andreas Krug.
\newblock On derived autoequivalences of {H}ilbert schemes and generalised
  {K}ummer varieties.
\newblock {\em Int.\ Math.\ Res.\ Not.}, pages 10680--10701, 2015.

\bibitem[KS15a]{KSosequiequi}
Andreas Krug and Pawel Sosna.
\newblock Equivalences of equivariant derived categories.
\newblock {\em J. Lond. Math. Soc. (2)}, 92(1):19--40, 2015.

\bibitem[KS15b]{KSosEnriques}
Andreas Krug and Pawel Sosna.
\newblock On the derived category of the {H}ilbert scheme of points on an
  {E}nriques surface.
\newblock {\em Selecta Math. (N.S.)}, 21(4):1339--1360, 2015.

\bibitem[OS11]{OSEnriques}
Keiji Oguiso and Stefan Schr{\"o}er.
\newblock Enriques manifolds.
\newblock {\em J. Reine Angew. Math.}, 661:215--235, 2011.

\bibitem[ST01]{ST}
Paul Seidel and Richard Thomas.
\newblock Braid group actions on derived categories of coherent sheaves.
\newblock {\em Duke Math. J.}, 108(1):37--108, 2001.

\bibitem[Voi07]{VoisinI}
Claire Voisin.
\newblock {\em Hodge theory and complex algebraic geometry. {I}}, volume~76 of
  {\em Cambridge Studies in Advanced Mathematics}.
\newblock Cambridge University Press, Cambridge, english edition, 2007.
\newblock Translated from the French by Leila Schneps.

\end{thebibliography}

\end{document}